\documentclass[letterpaper]{amsart}

\usepackage[letterpaper]{geometry}
\usepackage{amsmath, amsthm, amsfonts, amsbsy, thmtools, amssymb}

\usepackage{thm-restate}
\usepackage{hyperref} 
\usepackage[]{ytableau}
\usepackage[mathscr]{euscript}
\usepackage{stmaryrd}
\usepackage[all, arc, curve, frame]{xy} 
\usepackage[T1]{fontenc}
\usepackage{tikz}
\usetikzlibrary{arrows}

\numberwithin{equation}{section}

\makeatletter
\def\namedlabel#1#2{\begingroup
	#2%
	\def\@currentlabel{#2}%
	\phantomsection\label{#1}\endgroup
}
\makeatother

\usepackage{cleveref}

\SetSymbolFont{stmry}{bold}{U}{stmry}{m}{n}

\def\bfla{\begin{flalign}}
\def\efla{\end{flalign}}
\def\bea{\begin{aligned}}
\def\ena{\end{aligned}}
\def\bf*{\begin{flalign*}}
\def\ef*{\end{flalign*}}

\newtheorem{thm}{Theorem}[section]
\newtheorem{prop}[thm]{Proposition}
\newtheorem{lem}[thm]{Lemma}
\newtheorem{cor}[thm]{Corollary}

\theoremstyle{remark}
\newtheorem{rem}[thm]{Remark}
\theoremstyle{definition}
\newtheorem{definition}[thm]{Definition}

\title{Large Genus Asymptotics for Siegel--Veech Constants}

\author{Amol Aggarwal} 

\begin{document}

\begin{abstract}
	
	In this paper we consider the large genus asymptotics for two classes of Siegel--Veech constants associated with an arbitrary connected stratum $\mathcal{H} (\alpha)$ of Abelian differentials. The first is the saddle connection Siegel--Veech constant $c_{\text{sc}}^{m_i, m_j} \big( \mathcal{H} (\alpha) \big)$ counting saddle connections between two distinct, fixed zeros of prescribed orders $m_i$ and $m_j$, and the second is the area Siegel--Veech constant $c_{\text{area}} \big( \mathcal{H}(\alpha) \big)$ counting maximal cylinders weighted by area. By combining a combinatorial analysis of explicit formulas of Eskin--Masur--Zorich that express these constants in terms of Masur--Veech strata volumes, with a recent result for the large genus asymptotics of these volumes, we show that $c_{\text{sc}}^{m_i, m_j} \big( \mathcal{H} (\alpha) \big) = (m_i + 1) (m_j + 1) \big( 1 + o(1) \big)$ and $c_{\text{area}} \big( \mathcal{H}(\alpha) \big) = \frac{1}{2} + o(1)$, both as $|\alpha| = 2g - 2$ tends to $\infty$. The former result confirms a prediction of Zorich and the latter confirms one of Eskin--Zorich in the case of connected strata.

\end{abstract}

\maketitle 

\tableofcontents

\section{Introduction}

\label{Introduction}

\subsection{Siegel--Veech Constants}

Fix a positive integer $g > 1$, and let $\mathcal{H} = \mathcal{H}_g$ denote the moduli space of pairs $(X, \omega)$, where $X$ is a Riemann surface of genus $g$ and $\omega$ is a holomorphic one-form on $X$. Equivalently, $\mathcal{H}$ is the total space of the Hodge bundle over the moduli space $\mathcal{M}_g$ of complex curves of genus $g$; $\mathcal{H}$ is typically referred to as the \emph{moduli space of Abelian differentials}. 

For any $(X, \omega) \in \mathcal{H}$, the one-form $\omega$ has $2g - 2$ zeros (counted with multiplicity) on $X$. Thus, the moduli space of Abelian differentials can be decomposed as a disjoint union $\mathcal{H} = \bigcup_{\alpha \in \mathbb{Y}_{2g - 2}} \mathcal{H} (\alpha)$, where $\alpha = (m_1, m_2, \ldots , m_n)$ is ranged over all partitions\footnote{See Section \ref{Partitions} for our conventions and notation on partitions. In particular, $\mathbb{Y}_{2g - 2}$ denotes the set of partitions of size $2g - 2$.} of $2g - 2$, and $\mathcal{H} (\alpha) \subset \mathcal{H}$ denotes the moduli space of pairs $(X, \omega)$ where $X$ is again a Riemann surface of genus $g$ and $\omega$ is a holomorphic differential on $X$ with $n$ distinct zeros of multiplicities $m_1, m_2, \ldots , m_n$. These spaces $\mathcal{H} (\alpha)$ are orbifolds called \emph{strata}, which need not be connected. In fact, it was shown in \cite{CCMS} that $\mathcal{H} (\alpha)$ is connected for $g \ge 3$ if and only if $\alpha \ne (g - 1, g - 1)$ and at least one part of $\alpha$ is odd. 

The one-form $\omega$ induces a flat metric on $X$ away from a finite set of conical singularities (also called \emph{saddles}), which constitute the zeros $\{ p_1, p_2, \ldots , p_n \}$ of $\omega$; this makes $(X, \omega)$ into a translation surface. A \emph{saddle connection} on $(X, \omega)$ is a geodesic on $X$ connecting two saddles with no saddle in its interior, and a \emph{maximal cylinder} on $(X, \omega)$ is a Euclidean cylinder isometrically embedded in $X$ whose two boundaries are both unions of saddle connections. We will be interested in the (weighted) enumeration of saddle connections and maximal cylinders on a typical flat surface in some connected stratum of large genus. Due to scaling considerations, it will be convenient to restrict to the moduli space $\mathcal{H}_1 (\alpha) \subset \mathcal{H} (\alpha)$ of pairs $(X, \omega) \in \mathcal{H} (\alpha)$ such that $\frac{\textbf{i}}{2} \int_X \omega \wedge \overline{\omega} = 1$; this is the hypersurface of the stratum $\mathcal{H} (\alpha)$ consisting of $(X, \omega)$ where $\omega$ has area one. 

More specifically, let $L > 0$ be a real number. For any integers $1 \le i \ne j \le n$, let $N_{\text{sc}}^{p_i, p_j} (L) = N_{\text{sc}}^{p_i, p_j} \big( L; (X, \omega) \big)$ denote the number of saddle connections on $(X, \omega)$ connecting $p_i$ to $p_j$, with length at most $L$. Further let 
\begin{flalign*}
N_{\text{area}} (L) = N_{\text{area}} \big( L; (X, \omega) \big) = \displaystyle\frac{1}{A (X)} \displaystyle\sum_{w (C) \le L} A (C),
\end{flalign*}

\noindent where $C$ ranges over all maximal cylinders of $(X, \omega)$ of width at most $L$, and $A(X)$ and $A(C)$ denote the areas of $X$ and $C$, respectively. Stated alternatively, $N_{\text{area}} \big(L; (X, \omega) \big)$ counts maximal cylinders on $(X, \omega)$ of width at most $L$, weighted by their area proportion $\frac{A(C)}{A(X)}$. One might view a maximal cylinder as a ``thickening'' of a closed geodesic, in which case $N_{\text{area}} (L)$ can be viewed as a count for closed geodesics weighted by ``thickness.''

It was shown by Eskin--Masur in \cite{AFS} that, for a typical (by which we mean full measure subset with respect to the Masur--Veech volume; see \Cref{EstimatesVolume} for definitions) flat surface $(X, \omega)$ in some connected component of a stratum $\mathfrak{C} \subseteq \mathcal{H} (\alpha)$, the quantities $N_{\text{sc}} (L)$ and $N_{\text{area}} (L)$ grow quadratically in $L$ with asymptotics
\begin{flalign}
\label{c1c2}
c_{\text{sc}}^{m_i, m_j} (\mathfrak{C}) = \displaystyle\lim_{L \rightarrow \infty} \displaystyle\frac{N_{\text{sc}}^{p_i, p_j} \big( L; (X, \omega) \big)}{\pi L^2}; \qquad 
c_{\text{area}} (\mathfrak{C}) = \displaystyle\lim_{L \rightarrow \infty} \displaystyle\frac{N_{\text{area}} \big( L; (X, \omega) \big)}{\pi L^2}.
\end{flalign}

These two numbers $c_{\text{sc}}$ and $c_{\text{area}}$ fall into a class of quantities known as \emph{Siegel--Veech constants}, which had been studied in the earlier work \cite{M} of Veech; $c_{\text{area}}$ is sometimes specifically referred to as an \emph{area Siegel--Veech constant}. The quantity $c_{\text{sc}}$ does not depend on the specific choice of $p_i \ne p_j$ but rather on the orders of the zeros of $\omega$ at these two points. 

In addition to enumerating geometric phenomena, some Siegel--Veech constants also contain information about dynamics on the moduli space $\mathcal{H}_g$. To explain one such example, first recall that each stratum $\mathcal{H} (\alpha)$ admits an $\text{SL}_2 (\mathbb{R})$ action, under which an element $\textbf{A} \in \text{SL}_2 (\mathbb{R})$ acts on a translation surface $(X, \omega) \in \mathcal{H} (\alpha)$ by composing the local coordinate charts on $X$ with $\textbf{A}$. 

The \emph{Teichm\"{u}ller geodesic flow} on $\mathcal{H} (\alpha)$ is the action of the diagonal one-parameter subgroup $\left[ \begin{smallmatrix} e^t & 0 \\ 0 & e^{-t} \end{smallmatrix} \right]$ on this stratum. This flow lifts to any connected component of the Hodge bundle and, by Oseledets theorem, one can associate this flow with $2g$ \emph{Lyapunov exponents}, denoted by $\lambda_1 (\mathfrak{C}) \ge \lambda_2 (\mathfrak{C}) \ge \cdots \ge \lambda_{2g} (\mathfrak{C})$. These exponents are symmetric with respect to $0$, that is, $\lambda_i + \lambda_{2g - i + 1} = 0$ for each integer $i \in [1, 2g]$. We refer to the surveys \cite{TADIETFSB,TSOC,FS} and references therein for more information on the Teichm\"{u}ller geodesic flow and its applications in the theory of dynamical systems. 

It was shown as Theorem 1 in the work \cite{ERG} of Eskin--Kontsevich--Zorich that the sum of the first $g$ (the nonnegative) Lyapunov exponents associated with some connected component $\mathfrak{C}$ of a stratum $\mathcal{H} (\alpha)$, for some partition $\alpha = (m_1, m_2, \ldots , m_n)$ of $2g - 2$, can be expressed explicitly in terms of the area Siegel--Veech constant $c_{\text{area}} (\mathfrak{C})$ through the identity 
\begin{flalign}
\label{clambdasum} 
\displaystyle\sum_{i = 1}^g \lambda_i (\mathfrak{C}) = \displaystyle\frac{1}{12} \displaystyle\sum_{i = 1}^n \displaystyle\frac{m_i (m_i + 2)}{m_i + 1} + \displaystyle\frac{\pi^2}{3} c_{\text{area}} (\mathfrak{C}).
\end{flalign}

\noindent Thus, knowledge of the area Siegel--Veech constant of some connected stratum enables one to evaluate the sum of the associated positive Lyapunov exponents of the Teichm\"{u}ller geodesic flow.

\subsection{Large Genus Asymptotics} 

Although \cite{AFS} shows that the limits \eqref{c1c2} defining the above Siegel--Veech constants exist, it does not indicate how to evaluate them. This was done in the work \cite{PBC} of Eskin--Masur--Zorich, which evaluates these constants\footnote{For the area Siegel--Veech constant, one must combine the results of \cite{PBC} with one of Vorobets \cite{PGTS} (see also equation (2.16) of \cite{ERG}).} as a combinatorial sum involving the Masur--Veech volumes of (connected components) of strata. These strata volumes were evaluated through a (sort of intricate) algorithm of Eskin--Okounkov \cite{ANBCTV}, and a similar algorithm for evaluating the volumes of connected components of strata was later proposed by Eskin--Okounkov--Pandharipande \cite{TCBC}; the work of Chen--M\"{o}ller--Sauvaget--Zagier \cite{VSCC} provides an alternative way to access these volumes through a recursion.

Once it is known that these constants can in principle be determined, a question of interest is to understand how they behave as the genus $g$ tends to $\infty$. For example, in the similar context of Weil--Petersson volumes, such questions were investigated at length in \cite{GVRHSLG,LGAIMSC,LGAV}, and they were also considered in algebraic geometry to understand slopes of Teichm\"{u}ller curves in \cite{SSRCMS}. From the perspective of flat geometry, large genus asymptotics have also been studied in a number of recent works \cite{LGAVSD, QLGL,VSDCLG,TSGTMS,VSI}. 

Towards this direction, Eskin--Zorich posed a series of predictions (apparently in 2003, although the conjectures were not published until over a decade later in \cite{VSDCLG}) for the behavior of two geometric quantities associated with the strata $\mathcal{H} (\alpha)$, namely, their Masur--Veech volumes and the area Siegel--Veech constants of their connected components. The former prediction was based on numerical data provided by a program written by Eskin that implements the algorithm of Eskin--Okounkov to evaluate volumes of strata of genus $g \le 10$. This prediction was established in the cases of the principal and minimal strata (see below) in the works of Chen--M\"{o}ller--Zagier \cite{QLGL} and Sauvaget \cite{VSI}, respectively. It was later confirmed in general in \cite{LGAVSD}, and then subsequently an independent proof was proposed by Chen-M\"{o}ller-Sauvaget-Zagier in \cite{VSCC}. These results were recently used by Masur--Rafi--Randecker \cite{TSGTMS} to analyze the covering radius of a generic translation surface in an arbitrary stratum of large genus. 

The latter prediction was based on numerical data coming from two sources. The first was from combining the explicit formulas of Eskin--Masur--Zorich \cite{PBC} with the volumes tabulated by the above mentioned program to obtain explicit values of $c_{\text{area}} (\mathfrak{C})$ for connected components $\mathfrak{C} \subseteq \mathcal{H} (\alpha)$ of strata of genus $g \le 9$; the second, which seemed to be more useful in higher genus, was based on combining \eqref{clambdasum} with computer experiments approximating the Lyapunov exponents for the Teichm\"{u}ller geodesic flow. 

In any case, the prediction of Eskin--Zorich for the area Siegel--Veech constant states (see Main Conjecture 2 of \cite{VSDCLG}) that 
\begin{flalign}
\label{estimatearea}
\displaystyle\lim_{g \rightarrow \infty} \displaystyle\max_{\alpha \in \mathbb{Y}_{2g - 2}} \displaystyle\max_{\mathfrak{C} \subseteq \mathcal{H} (\alpha)} \left| c_{\text{area}} ( \mathfrak{C}) - \displaystyle\frac{1}{2} \right| = 0,
\end{flalign}

\noindent where $\mathfrak{C}$ ranges over all non-hyperelliptic connected components\footnote{The behavior of the area Siegel--Veech constant of the hyperelliptic connected component of a stratum is considerably different; see Corollary 1 of \cite{ERG} or Remark 5 of \cite{VSDCLG}.} of the stratum $\mathcal{H} (\alpha)$. Observe here that \eqref{estimatearea} predicts that $c_{\text{area}} (\mathfrak{C} )$ should converge to $\frac{1}{2}$ as the genus corresponding to $\mathfrak{C}$ tends to $\infty$, independently of $\mathfrak{C}$.

Before this work, the asymptotic \eqref{estimatearea} had been verified in two cases. First, the work of Chen--M\"{o}ller--Zagier \cite{QLGL} established \eqref{estimatearea} if $\mathcal{H} (m)$ is the \emph{principal stratum}, that is, when $m = 1^{2g - 2}$; this corresponds to the stratum in which all zeros of the holomorphic differential $\omega$ are distinct. By analyzing a certain cumulant generating function, they establish \eqref{estimatearea} as Theorem 19.4 of \cite{QLGL}. Second, the work of Sauvaget \cite{VSI} established \eqref{estimatearea} if $\mathcal{H} (m)$ is the \emph{minimal stratum}, that is, when $m = (2g - 2)$; this corresponds to the stratum in which $\omega$ has one zero with multiplicity $2g - 2$. Through an analysis of Hodge integrals on the moduli space of curves (based on his earlier work \cite{CCSD}), he establishes \eqref{estimatearea} in the case of the minimal stratum $\mathcal{H} (2g - 2)$ as Theorem 1.9 of \cite{VSI}. 

\begin{rem} 

\label{connected} 

Although we defined Siegel--Veech constants on connected components of strata, one can define suitable analogs of them on all of a disconnected stratum through a weighted sum (see Remark 1.10 of \cite{VSI}). Sauvaget's result applies to this variant of the Siegel--Veech constant (since $\mathcal{H} (2g - 2)$ is disconnected).

\end{rem}

The lower bound $c_{\text{area}} ( \mathfrak{C}) \ge \frac{1}{2} + o (1)$ was also established for all connected strata by Zorich in an appendix to \cite{LGAVSD} (see Theorem 1 there). 
	
Concerning the Siegel--Veech constant $c_{\text{sc}}^{m_i, m_j} \big( \mathcal{H} (\alpha) \big)$ counting saddle connections, it was predicted by Zorich in an appendix to \cite{LGAVSD} that
\begin{flalign} 
\label{asymptoticcsc}
c_{\text{sc}}^{m_i, m_j} \big(\mathcal{H} (\alpha) \big) = (m_i + 1) (m_j + 1) \big( 1 + o (1) \big),
\end{flalign}

\noindent for any connected stratum $\mathcal{H} (\alpha)$. 

There, he showed as Corollary 1 of \cite{LGAVSD}  that the lower bound $c_{\text{sc}}^{m_i, m_j} \big(\mathcal{H} (\alpha) \big) \ge (m_i + 1) (m_j + 1) \big( 1 + o (1) \big)$ holds for all strata. In fact, he showed the contribution of \emph{multiplicity one} saddle connections to $c_{\text{sc}}^{m_i, m_j} \big(\mathcal{H} (\alpha) \big)$, namely those that do not have the same holonomy vector as a different saddle connection, is $(m_i + 1) (m_j + 1) \big(1 + o(1) \big)$ and then predicted that the contribution coming from the remaining saddle connections (that is, those of higher multiplicities) should become negligible in the large genus limit. He showed this to be true for the principal stratum $\alpha = 1^{2g - 2}$ as Corollary 2 of \cite{LGAVSD}, thereby establishing asymptotic \eqref{asymptoticcsc} in this case. 
 
 Let us conclude this section by mentioning that, subsequent to the appearance of this article, Chen--M\"{o}ller--Sauvaget--Zagier \cite{VSCC} proposed an independent and very different proof of the Eskin-Zorich prediction \eqref{estimatearea}, by interpreting the area Siegel--Veech constant as an intersection number on the moduli space of Abelian differentials. In that work, they also consider a modification of the Siegel--Veech constant $c_{\text{sc}}^{m_i, m_j} \big (\mathcal{H} (\alpha) \big)$ defined above, which differs from it in that families of saddle connections sharing the same holonomy vector are identified. They show that this smaller Siegel--Veech constant is in fact equal to $(m_i + 1) (m_j + 1)$, independently of the genus of the stratum.

\subsection{Results}

In this paper we confirm the asymptotics \eqref{estimatearea} and \eqref{asymptoticcsc} for all connected strata. The following theorem establishes the former asymptotic, on $c_{\text{sc}}$.

	\begin{thm} 
		
	\label{sgestimate}
	
	There exists a constant $C > 0$ such that the following holds. For any integer $g > 2$; partition $\alpha = (m_1, m_2, \ldots , m_n)$ of $2g - 2$ such that the stratum $\mathcal{H} (\alpha)$ is nonempty and connected; and integers $1 \le i \ne j \le n$, we have that  
	\begin{flalign*}
	\big| c_{\text{\emph{sc}}}^{m_i, m_j} (\alpha) - (m_i + 1) (m_j + 1) \big| \le \displaystyle\frac{C (m_i + 1)(m_j + 1)}{g}.
	\end{flalign*} 
	
\end{thm}

The following theorem establishes the latter asymptotic, on $c_{\text{area}}$.

\begin{thm} 
	
	\label{constantareaasymptotic} 
	
	There exists a constant $C > 0$ such that the following holds. For any integer $g > 2$ and partition $\alpha$ of $2g - 2$ such that the stratum $\mathcal{H} (\alpha)$ is nonempty and connected, we have that 
	\begin{flalign*}
	\left| c_{\text{\emph{area}}} \big( \mathcal{H} (\alpha) \big) - \displaystyle\frac{1}{2} \right| \le \displaystyle\frac{C}{g}. 
	\end{flalign*}
	
\end{thm} 

As a consequence of \Cref{constantareaasymptotic} and \eqref{clambdasum}, we deduce the following approximation for the sum of Lyapunov exponents of the Teichm\"{u}ller geodesic flow (see also equation (5) of \cite{VSDCLG}). 

\begin{cor} 
	
	\label{exponentasymptotic} 
	
	There exists a constant $C > 0$ such that the following holds. For any integer $g > 2$ and partition $\alpha = (m_1, m_2, \ldots , m_n)$ of $2g - 2$ such that the stratum $\mathcal{H} (\alpha)$ is nonempty and connected, we have that 
	\begin{flalign*}
	\left| \displaystyle\sum_{i = 1}^g \lambda_i \big( \mathcal{H} (\alpha) \big) - \displaystyle\frac{1}{12} \displaystyle\sum_{i = 1}^n \displaystyle\frac{m_i (m_i + 2)}{m_i + 1} - \displaystyle\frac{\pi^2}{6} \right| \le \displaystyle\frac{C}{g}. 
	\end{flalign*}
	
\end{cor} 

The proofs of \Cref{sgestimate} and \Cref{constantareaasymptotic} will appear in \Cref{ConstantAsymptotic} and \Cref{AreaConstant}, respectively.  Both are based on a combination of a combinatorial analysis of the formulas in \cite{PBC} that express the relevant Siegel--Veech constants of connected strata through the Masur--Veech volumes of (possibly different) strata, with the recent results of \cite{LGAVSD} that analyze the large genus asymptotics for these volumes.\footnote{Using Formula 14.5 of \cite{PBC}, is plausible that our methods for the area Siegel--Veech constant can also be extended to apply to all strata, where we define this constant on disconnected strata through a suitably weighted sum (as in \Cref{connected}). However, we will not pursue this here.} However, let us mention that the formulas for the Siegel--Veech constants from \cite{PBC} are a bit intricate; they involve a number of terms that can potentially grow exponentially with the genus $g$ associated with the stratum. Still, as in \cite{LGAVSD}, we will show that each of these sums is dominated by a single term and that the remaining terms can be viewed as negligible. 

The remainder of this article is organized as follows. In \Cref{Estimates}, we recall some notation on partitions, state the results from \cite{LGAVSD} on large genus asymptotics of Masur--Veech volumes, and provide some preliminary combinatorial estimates. We then establish \Cref{sgestimate} in \Cref{ConstantAsymptotic} and \Cref{constantareaasymptotic} in \Cref{AreaConstant}.

\subsection*{Acknowledgments}

The author heartily thanks Anton Zorich for many stimulating conversations, valuable encouragements, and enlightening explanations, and also for kindly translating the results of \cite{PBC} into the combinatorial formulas described in \Cref{IdentityConstant1} and \Cref{IdentityConstantArea}. The author is also grateful to Dawei Chen for helpful comments and discussions, as well as to the anonymous referee for detailed suggestions on the initial draft of this paper. This work was partially supported by the NSF Graduate Research Fellowship under grant number DGE1144152.

\section{Miscellaneous Preliminaries} 

\label{Estimates}

In this section we recall some notation and estimates that will be used throughout this paper. In particular, Section \ref{Partitions} will set some notation on partitions and set partitions; \Cref{EstimatesVolume} will recall some notation on Masur--Veech volumes and results from \cite{LGAVSD} about their large genus asymptotics; and Section \ref{Estimates1} will collect several combinatorial estimates to be applied later.

\subsection{Notation} 

\label{Partitions}

A \emph{partition} $\lambda = (\lambda_1, \lambda_2, \ldots , \lambda_k)$ is a finite, nondecreasing sequence of positive integers. The numbers $\lambda_1, \lambda_2, \ldots , \lambda_k$ are called the \emph{parts} of $\lambda$; the number of parts $\ell (\lambda) = k$ is called the \emph{length} of $\lambda$; and the sum of the parts $|\lambda| = \sum_{i = 1}^k \lambda_i$ is called the \emph{size} of $\lambda$. For each integer $n \ge 0$, let $\mathbb{Y}_n$ denote the set of partitions of size $n$.

A \emph{composition} of $n$ of length $k$ is an (ordered) $k$-tuple $(j_1, j_2, \ldots , j_k)$ of positive integers that sum to $n$. We denote the set of compositions of $n$ of length $k$ by $\mathcal{C}_n (k)$. We also denote the set of \emph{nonnegative compositions} of $n$ of length $k$, that is, the set of (ordered) $k$-tuples $(j_1, j_2, \ldots , j_k)$ of nonnegative integers that sum to $n$, by $\mathcal{G}_n (k)$. Observe that
\begin{flalign}
\label{ynkestimate}
\big| \mathcal{C}_n (k) \big| = \binom{n - 1}{k - 1}; \qquad \big| \mathcal{G}_n (k) \big| = \binom{n + k - 1}{k - 1}.
\end{flalign}

In addition to discussing partitions, we will also consider set partitions. For any finite set $S$, a \emph{set partition} $\alpha = \big( \alpha^{(1)}, \alpha^{(2)}, \ldots , \alpha^{(k)} \big)$ of $S$ is a sequence of mutually disjoint subsets $\alpha^{(i)} \subseteq S$ such that $\bigcup_{i = 1}^k \alpha^{(i)} = S$; these subsets $\alpha^{(i)}$ are called the \emph{components} of $\alpha$. The \emph{length} $\ell (\alpha) = k$ of $\alpha$ denotes the number of components of $\alpha$. For the purposes of this article, set we will distinguish two set partitions consisting of the same components but in a different order. For instance, if $S = \{1, 2, 3, 4 \}$, then we view the set partitions $\big( \{1, 2 \}, \{ 3, 4 \} \big)$ and $\big( \{3, 4\}, \{1, 2 \} \big)$ as distinct. 

For any positive integers $n$ and $k$, let $\mathfrak{P}_n$ denote the family of set partitions of $\{1, 2, \ldots , n \}$, and let $\mathfrak{P}_{n; k}$ denote the family of set partitions of $\{ 1, 2, \ldots , n \}$ of length $k$. Furthermore, for any composition $A = (A_1, A_2, \ldots, A_k) \in \mathcal{C}_n (k)$, let $\mathfrak{P} (A) = \mathfrak{P} (A_1, A_2, \ldots , A_k) = \mathfrak{P}_{n; k} (A_1, A_2, \ldots , A_k)$ denote the family of set partitions $\big(\alpha^{(1)}, \alpha^{(2)}, \cdots , \alpha^{(k)} \big)$ of $\{1, 2, \ldots , n \}$ such that $\ell \big(\alpha^{(i)} \big) = A_i$, meaning that $ \alpha^{(i)}$ has $A_i$ elements, for each $1 \le i \le k$. Observe in particular that 
\begin{flalign}
\label{aipksize}
\big| \mathfrak{P} (A) \big| = \binom{n}{A_1, A_2, \ldots , A_k}; \qquad \mathfrak{P}_{n, k} = \bigcup_{A \in \mathcal{C}_n (k)} \mathfrak{P} (A). 
\end{flalign}

\subsection{Large Genus Asymptotics for Masur--Veech Volumes}

\label{EstimatesVolume} 

In this section we recall the definition of Masur--Veech volumes of strata of Abelian differentials, and we also recall a result of \cite{LGAVSD} that evaluates their large genus limits.

To that end, fix an integer $g > 1$. There exists a measure on $\mathcal{H} = \mathcal{H}_g$ (or, equivalently, on each of its strata $\mathcal{H} (\alpha)$, for any $\alpha \in \mathbb{Y}_{2g - 2}$) that is invariant with respect to the action of $\text{SL}_2 (\mathbb{R})$ on $\mathcal{H}$. This measure can be defined as follows. 

Let $\alpha = (m_1, m_2, \ldots , m_n) \in \mathbb{Y}_{2g - 2}$, let $(X, \omega) \in \mathcal{H} (\alpha)$ be a pair in the stratum corresponding to $\alpha$, and define $k = 2g + n - 1$. Denote the zeros of $\omega$ by $p_1, p_2, \ldots , p_n \in X$, and let $\gamma_1, \gamma_2, \ldots , \gamma_k$ denote a basis of the relative homology group $H_1 \big( X, \{ p_1, p_2, \ldots , p_n \}, \mathbb{Z} \big)$. Consider the \emph{period map} $\Phi: \mathcal{H} (\alpha) \rightarrow \mathbb{C}^k$ obtained by setting $\Phi (X, \omega) = \big( \int_{\gamma_1} \omega, \int_{\gamma_2} \omega, \ldots , \int_{\gamma_k} \omega \big)$. It can be shown that the period map $\Phi$ defines a local coordinate chart (called \emph{period coordinates}) for the stratum $\mathcal{H} (\alpha)$. Pulling back the Lebesgue measure on $\mathbb{C}^k$ yields a measure $\nu$ on $\mathcal{H} (\alpha)$, which is quickly verified to be independent of the basis $\{ \gamma_i \}$ and invariant under the action of $\text{SL}_2 (\mathbb{R})$ on $\mathcal{H} (\alpha)$. 

Let $\nu_1$ denote the measure induced by $\nu$ on the hypersurface $\mathcal{H}_1 (\alpha) \subset \mathcal{H} (\alpha)$ consisting of area one Abelian differentials; we abbreviate $\nu_1 \big( \mathcal{H} (\alpha) \big) = \nu_1 \big( \mathcal{H}_1 (\alpha) \big)$. If $|\alpha|$ is odd, then the stratum $\mathcal{H} (\alpha)$ does not exist, and we instead define $\nu_1 \big( \mathcal{H} (\alpha) \big) = 0$. It was established independently by Masur \cite{ETMF} and Veech \cite{MTSIEM} that $\nu_1$ is ergodic on each connected component of $\mathcal{H}_1 (\alpha)$ under the action of $\text{SL}_2 (\mathbb{R})$ and that the volume $\nu_1 \big( \mathcal{H} (\alpha) \big)$ is finite for each $\alpha$. This volume $\nu_1 \big( \mathcal{H} (\alpha) \big)$ is called the \emph{Masur--Veech volume} of the stratum indexed by $\alpha$. 

Although the finiteness of the Masur--Veech volumes was established in 1982 \cite{ETMF,MTSIEM}, it was nearly two decades until mathematicians produced general ways of finding these volumes explicitly. One of the earlier exact evaluations of these volumes seems to have appeared in the paper \cite{SVMS} of Zorich (although he mentions that the idea had been independently suggested by Eskin--Masur and Kontsevich--Zorich two years earlier), in which he evaluates $\nu_1 \big( \mathcal{H} (\alpha) \big)$ for some partitions $m$ corresponding to small values of the genus $g$. Through a very different method, based on the representation theory of the symmetric group and asymptotic Hurwitz theory, Eskin--Okounkov \cite{ANBCTV} proposed a general algorithm that, given an integer $g > 1$ and $\alpha = (\alpha_1, \alpha_2, \ldots , \alpha_n) \in \mathbb{Y}_{2g - 2}$, determines the volume of the stratum $\nu_1 \big( \mathcal{H} (\alpha) \big)$. 

Based on the numerical data provided by a program implemented by Eskin to evaluate these volumes, Eskin and Zorich posed a precise prediction (see Conjecture 1 and equations (1) and (2) of \cite{VSDCLG}) for the behavior of the Masur--Veech volumes $\nu_1 \big( \mathcal{H} (\alpha) \big)$ in the large genus limit, as $|\alpha| = 2g - 2$ tends to $\infty$. Through a combinatorial analysis of the Eskin--Okounkov algorithm, this prediction was established as Theorem 1.4 of \cite{LGAVSD}, which states the following.

\begin{prop}[{\cite[Theorem 1.4]{LGAVSD}}]
	
	\label{volumeestimateg} 
	
	For any positive integer $g$ and $n$-tuple $\alpha = (m_1, m_2, \ldots , m_n)$ of positive integers that sum to $2g - 2$, we have that
	\begin{flalign}
	\label{volumeestimate1h1}
	\displaystyle\frac{4}{\prod_{i = 1}^n (m_i + 1)}  \left( 1 - \displaystyle\frac{2^{2^{200}}}{g}  \right) \le \nu_1 \big( \mathcal{H} (\alpha) \big) \le \displaystyle\frac{4}{\prod_{i = 1}^n (m_i + 1)}  \left( 1 +  \displaystyle\frac{2^{2^{200}}}{g}  \right).
	\end{flalign}
	
\end{prop}

A consequence of \eqref{volumeestimate1h1} is the existence of a constant $R > 2$ such that 
\begin{flalign}
\label{volumeestimate}
\displaystyle\frac{1}{R} \le \nu_1 \big( \mathcal{H} (\alpha) \big) \displaystyle\prod_{j = 1}^n (m_j + 1) \le R.
\end{flalign}

These estimates will be useful for the proofs of \Cref{sgestimate} and \Cref{constantareaasymptotic}, since the Siegel--Veech constants in those two theorems can be expressed explicitly in terms of the Masur--Veech volumes of various strata (see \Cref{IdentityConstant1} and \Cref{IdentityConstantArea} below).

\subsection{Estimates}

\label{Estimates1}

In this section we collect several estimates that will be used at various points throughout this paper. We will repeatedly use the bounds 
\begin{flalign} 
\label{2ll}
\begin{aligned}
k \le & 2^{k - 1};  \qquad \displaystyle\frac{2 k^{k + 1 / 2}}{e^k}  \le k! \le \displaystyle\frac{4 k^{k + 1 / 2}}{e^k},
\end{aligned}
\end{flalign}

\noindent which hold for any integer $k \ge 0$. Next, we have the following multinomial coefficient estimate. 

\begin{lem}
	
\label{sumaijaiestimate}

Let $n$ and $r$ be positive integers; also let $\{ A_i \}$ and $\{ A_{i, j} \}$, for $1 \le i \le n$ and $1 \le j \le r$, be sets of nonnegative integers such that $\sum_{j = 1}^r A_{i, j} = A_i$ for each $i$. Then,
\begin{flalign}
\label{aijai}
\displaystyle\prod_{i = 1}^n \binom{A_i}{A_{i, 1}, A_{i, 2}, \ldots , A_{i, r}} \le \binom{\sum_{j = 1}^n A_i}{\sum_{i = 1}^n A_{i, 1}, \sum_{i = 1}^n A_{i, 2}, \ldots , \sum_{i = 1}^n A_{i, r}}.
\end{flalign}

\end{lem}

\begin{proof}
	
	For each $0 \le k \le n$, let $T_k = \sum_{j = 1}^k A_j$ (with $T_0 = 0$). Define the sets $\mathcal{S} = \{ 1, 2, \ldots , T_n \}$ and $\mathcal{S}_i = \big\{ T_{i - 1} + 1, T_{i - 1}  + 2, \ldots , T_i \big\}$ for each $1 \le i \le n$. Then, by the first identity in \eqref{aipksize}, the left side of \eqref{aijai} counts the number of ways to partition the $\mathcal{S}_i$ into $r$ mutually disjoint subsets $\mathcal{S}_{i, 1},  \mathcal{S}_{i, 2}, \ldots \mathcal{S}_{i, r}$ consisting of $A_{i, 1}, A_{i, 2}, \ldots , A_{i, r}$ elements, respectively. Similarly, the right side of \eqref{aijai} counts the number of ways to partition $\mathcal{S}$ into $r$ mutually disjoint subsets $\mathcal{S}^{(1)}, \mathcal{S}^{(2)}, \ldots , \mathcal{S}^{(r)}$ consisting of $\sum_{i = 1}^n A_{i, 1}, \sum_{j = 1}^n A_{i, 2}, \ldots,  \sum_{j = 1}^n A_{i, r}$ elements, respectively.
	
	Any partition of the former type gives rise to a unique partition of the latter type by setting $\mathcal{S}^{(k)} = \bigcup_{i = 1}^n \mathcal{S}_{i, k}$ for each $1 \le k \le r$. Thus, the left side of \eqref{aijai} is at most equal to the right side of \eqref{aijai}. 
\end{proof}

The following estimate bounds products of factorials and will be used several times in the proofs of \Cref{sgestimate} and \Cref{constantareaasymptotic}.

\begin{lem}
	
	\label{aibiproduct} 
	
	Let $T$, $U$, $V$, and $d$ be nonnegative integers and $r$ be a positive integer. Further let $A = (A_1, A_2, \ldots , A_r) \in \mathcal{G}_T (r)$; $B = (B_1, B_2, \ldots,  B_r) \in \mathcal{G}_U (r)$; and $C = (C_1, C_2, \ldots , C_r) \in \mathcal{G}_V (r)$ be nonnegative compositions of lengths $r$ and sizes $T$, $U$, and $V$, respectively. If $A_i \ge C_i$ for each $i \in [1, r]$, then 
	\begin{flalign}
	\label{aibidtuvproduct}
	\displaystyle\prod_{i = 1}^r (A_i + B_i + d)! \le \displaystyle\frac{(T + U + d)!}{(V + U + d)!} \displaystyle\prod_{i = 1}^r (C_i + B_i + d)!.  
	\end{flalign}
	
\end{lem}

\begin{proof} 
	
	This lemma follows from the fact that 
	\begin{flalign*}
	\displaystyle\prod_{i = 1}^r \displaystyle\frac{(A_i + B_i + d)!}{(C_i + B_i + d)!} = \displaystyle\prod_{i = 1}^r \displaystyle\prod_{j = 0}^{A_i - C_i - 1} (C_i + B_i + d + 1 + j)! & \le  \displaystyle\prod_{j = 0}^{T - V - 1} (V + U + d + 1 + j)! \\
	& = \displaystyle\frac{(T + U + d)!}{(V + U + d)!}, 
	\end{flalign*}
	
	\noindent where in the inequality we used the facts that $\sum_{i = 1}^r (A_i - C_i) = T - V$ and that $B_i + C_i + d + 1 + j \le U + V + d + 1 + j'$ for any $j' \ge j$, which holds since $B_i \le U$ and $C_i \le V$ for each $i \in [1, r]$. 
\end{proof}

Next, we require the following two lemmas, which are additional estimates on products of factorials. The first appeared as Lemma 2.5 of \cite{LGAVSD}; the proof of the second (given by \Cref{aicisum1} below) is very similar to that of \Cref{aicisum} and is therefore omitted.

\begin{lem}[{\cite[Lemma 2.5]{LGAVSD}}]
	
	\label{aicisum}
	
	Let $k \ge 1$ and $C_1, C_2, \ldots , C_k$ be nonnegative integers with $C_1 = \max_{1 \le i \le k} C_i$. Fix some nonnegative integer $N$, and let $A_1, A_2, \ldots , A_k$ be nonnegative integers such that $\sum_{i = 1}^k A_i = N$. Then, 
	\begin{flalign}
	\label{aici1}
	\displaystyle\prod_{i = 1}^k (A_i + C_i)! \le (N + C_1)! \displaystyle\prod_{i = 2}^k C_i!. 
	\end{flalign}

\end{lem}

\begin{lem}
	
	\label{aicisum1} 
	
	Fix positive integers $k$ and $N$. Let $A_1, A_2, \ldots , A_k$ be nonnegative integers such that $\sum_{i = 1}^k A_i = N$. If $\max_{1 \le i \le k} A_i \le N - 1$, then 
	\begin{flalign}
	\label{aici2}
	\displaystyle\prod_{i = 1}^k (A_i + 1)! \le 2 N!; \qquad \displaystyle\prod_{i = 1}^k (A_i + 2)! \le 2^{k - 2} 6 (N + 1)!; \qquad \displaystyle\prod_{i = 1}^k (A_i + 5)! \le 5^{k - 2} 720 (N + 4)!. 
	\end{flalign}
	
	\noindent Furthermore, if $N \ge 2$ and $\max_{1 \le j \le k} A_j \le N - 2$, then 
	\begin{flalign}
	\label{aici3}
	\displaystyle\prod_{i = 1}^k (A_i + 4)! \le 4^{k - 2} 720 (N + 2)!. 
	\end{flalign}

\end{lem}

We conclude this section with the following two combinatorial estimates, which bound sums of products of factorials.

\begin{lem}
	
\label{2sumcompositions} 

Fix a positive integer $r$ and nonnegative integers $U$ and $V$. We have that 
\begin{flalign}
\label{abgurgvr}
\displaystyle\sum_{A \in \mathcal{G}_U (r)} \displaystyle\sum_{B \in \mathcal{G}_V (r)} \displaystyle\prod_{i = 1}^r (A_i + B_i + 3)! \le 2^{11r} (U + V + 3)!, 
\end{flalign}

\noindent where, on the left side of \eqref{abgurgvr}, we denoted the nonnegative compositions $A = (A_1, A_2, \ldots , A_r) \in \mathcal{G}_U (r)$ and $B = (B_1, B_2, \ldots , B_r) \in \mathcal{G}_V (r)$.
	
\end{lem} 

\begin{proof} 

We begin by expressing the left side of \eqref{abgurgvr}, which originally appears as a sum over two compositions, as a sum over one composition. To that end, for any $C = (C_1, C_2, \ldots , C_r) \in \mathcal{G}_{U + V} (r)$, let $N_C (U, V)$ denote the number of pairs $(A, B)$ of nonnegative compositions $A = (A_1, A_2, \ldots , A_r) \in \mathcal{G}_U (r)$ and $B = (B_1, B_2, \ldots , B_r) \in \mathcal{G}_V (r)$ such that $C_i = A_i + B_i$ for each $i \in [1, r]$. Then, 
\begin{flalign}
\label{abproduct1n}
\displaystyle\sum_{A \in \mathcal{G}_U (r)} \displaystyle\sum_{B \in \mathcal{G}_V (r)} \displaystyle\prod_{i = 1}^p (A_i + B_i + 3)! & = \displaystyle\sum_{C \in \mathcal{G}_{U + V} (r)} N_C (U, V) \displaystyle\prod_{i = 1}^r (C_i + 3)!. 
\end{flalign}

Now for each positive integer $m$, let $\mathcal{G}_{U + V} (r; m)$ denote the set of nonnegative compositions $C = (C_1, C_2, \ldots , C_r) \in \mathcal{G}_{U + V} (r)$ such that $\max_{1 \le i \le r} C_i \le m$. Define $\mathcal{C}_{U + V} (s; m)$ similarly (consisting of positive compositions satisfying the analogous property) for any integer $s \in [1, r]$.  In particular, $\mathcal{G}_{U + V} (r; U + V) = \mathcal{G}_{U + V} (r)$. 

Observe that there are at most $r$ compositions $C \in \mathcal{G}_{U + V} (r) \setminus \mathcal{G}_{U + V} (r; U + V - 1)$, which must have $r - 1$ parts equal to $0$ and one part equal to $U + V$ (since their size is $U + V$ and their largest part is larger than $U + V - 1$). For any such $C$, $N_C (U, V) = 1$, since if $j \in [1, r]$ is such that $C_j = U + V$ then we must have $A_j = U$, $B_j = V$, and $A_i = 0 = B_i$ for any $i \ne j$. Thus, the contribution coming from nonnegative compositions $C \in \mathcal{G}_{U + V} (r) \setminus \mathcal{G}_{U + V} (r; U + V - 1)$ to the right side of \eqref{abproduct1n} is bounded above by $r 3^{r - 1} (U + V + 3)!$. 

Therefore,   
\begin{flalign*}
\displaystyle\sum_{A \in \mathcal{G}_U (r)} \displaystyle\sum_{B \in \mathcal{G}_V (r)} \displaystyle\prod_{i = 1}^r (A_i + B_i + 3)! & \le r 3^r (U + V + 3)! + \displaystyle\sum_{C \in \mathcal{G}_{U + V} (r; U + V - 1)} N_C (U, V) \displaystyle\prod_{i = 1}^r (C_i + 3)!. 
\end{flalign*}

Next, due to the facts that $N_C (U, V) \le \prod_{i = 1}^r (C_i + 1)$ (since there are at most $C_i + 1$ choices for $A_i \in [0, C_i]$, for each $i \in [1, r]$, and $B$ is determined by $A$) and $(C_i + 1) (C_i + 3)! \le (C_i + 4)!$, it follows that 
\begin{flalign}
\label{abproduct2n}
\displaystyle\sum_{A \in \mathcal{G}_U (r)} \displaystyle\sum_{B \in \mathcal{G}_V (r)} \displaystyle\prod_{i = 1}^r (A_i + B_i + 3)! & \le r 3^r (U + V + 3)! + \displaystyle\sum_{C \in \mathcal{G}_{U + V} (r; U + V - 1)} \displaystyle\prod_{i = 1}^r (C_i + 4)!. 
\end{flalign}

Now, in order to bound the sum appearing on the right side of \eqref{abproduct2n}, observe that any nonnegative composition $C = (C_1, C_2, \ldots , C_r) \in \mathcal{G}_{U + V} (r; w)$ can be associated with a subset $\mathcal{I} \subseteq [1, r]$ and a (positive) composition $D = (D_1, D_2, \ldots , D_s) \in \mathcal{C}_{U + V} (s; w)$, where $s = r - |\mathcal{I}|$. Here, $\mathcal{I}$ is defined to be the set such that $C_i = 0$ if and only if $i \in \mathcal{I}$, and $D = C \setminus \{ C_i \}_{i \in \mathcal{I}}$.

Since there are at most $2^r$ possibilities for $\mathcal{I}$, and since $|\mathcal{I}| \le r - 2$ (so that $s \ge 2$) for any $C \in \mathcal{G}_{U + V} (s; U + V - 1)$, it follows that 
\begin{flalign} 
\label{abproduct3n} 
\displaystyle\sum_{C \in \mathcal{G}_{U + V} (r; U + V - 1)} \displaystyle\prod_{i = 1}^r (C_i + 4)! & \le 2^r \displaystyle\max_{2 \le s \le r} \displaystyle\sum_{D \in \mathcal{C}_{U + V} (s; U + V - 1)} \displaystyle\prod_{i = 1}^r (D_i + 4)!. 
\end{flalign}

 Now, observe that $\mathcal{C}_{U + V} (s) = \mathcal{C}_{U + V} (s; U + V - s + 1)$ since all parts of any composition in $\mathcal{C}_{U + V} (s)$ are all positive. Furthermore, any composition in $\mathcal{C}_{U + V} (s) \setminus  \mathcal{C}_{U + V} (s; U + V - s)$ has one part equal to $U + V - s + 1$ and the remaining parts all equal to $1$. There are at most $s \le r$ such compositions, and so the total contribution of these compositions to the sum on the right side of \eqref{abproduct1n} is at most $s 5^{s - 1} (U + V - s + 5)! \le s 5^{s - 1} (U + V + 3)!$, since $s \ge 2$. Thus, 
\begin{flalign*}
 \displaystyle\sum_{D \in \mathcal{C}_{U + V} (s)} \displaystyle\prod_{i = 1}^s (D_i + 4)! & \le  s 5^{s - 1} (U + V + 3)!  + \displaystyle\sum_{D \in \mathcal{C}_{U + V} (s; U + V - s)} \displaystyle\prod_{i = 1}^s (D_i + 4)!.
\end{flalign*}

\noindent Applying the third estimate in \eqref{aici2} with the $k$ there equal to $s$ here; the $A_i$ there equal to the $D_i - 1$; and the $N$ there equal to $U + V - s$ here, we deduce that $\prod_{i = 1}^s (D_i + 4)! \le 5^{s - 2} 720 (U + V + 4 - s)!$. Thus, 
\begin{flalign}
\label{abproduct4n} 
\begin{aligned} 
\displaystyle\sum_{D \in \mathcal{C}_{U + V} (s)} \displaystyle\prod_{i = 1}^s (D_i + 4)! & \le  s 5^{s - 1} (U + V + 5 - s)!  +  \displaystyle\sum_{D \in \mathcal{C}_{U + V} (s; U + V - s)} 5^{s + 3} (U + V + 4 - s)! \\
& \le  s 5^{s - 1} (U + V + 5 - s)!  + \binom{U + V - 1}{s - 1} 5^{s + 3} (U + V + 4 - s)! \\
& \le  s 5^{s - 1} (U + V + 5 - s)!  + \displaystyle\frac{5^{s + 3}  (U + V + 3)!}{(s - 1)!},
\end{aligned}
\end{flalign}

\noindent where in the second estimate we used the fact that $\big| \mathcal{C}_{U + V} (s; U + V - s) \big| \le \big| \mathcal{C}_{U + V} (s) \big| = \binom{U + V - 1}{s - 1}$, which holds in view of the first identity in \eqref{ynkestimate}. Now the lemma follows from inserting \eqref{abproduct3n} and \eqref{abproduct4n} into \eqref{abproduct2n} and using the first estimate in \eqref{2ll}, as well as the fact that $s \ge 2$. 
\end{proof}

\begin{lem}
	
	\label{sumcompositions}

	Fix a positive integer $r$ and nonnegative integers $U$, $V$, and $W$. We have that
	\begin{flalign}
	\label{estimatesumcompositions1} 
	\begin{aligned}
	\displaystyle\sum_{A \in \mathcal{G}_U (r)} \displaystyle\sum_{B \in \mathcal{G}_V (r)} \displaystyle\sum_{C \in \mathcal{G}_W (r)} \displaystyle\prod_{i = 1}^r (A_i + B_i + C_i + 1)! \le 2^{8r + 9} (U + V + W + 1)!,
	\end{aligned} 
	\end{flalign} 
	
	\noindent where, on the left side of \eqref{estimatesumcompositions1}, we denoted the nonnegative compositions $A = (A_1, A_2, \ldots , A_r) \in \mathcal{G}_U (r)$; $B = (B_1, B_2, \ldots , B_r) \in \mathcal{G}_V (r)$; and $C = (C_1, C_2, \ldots , C_r) \in \mathcal{G}_W (r)$.
\end{lem}

\begin{proof}
	
	We proceed as in the proof of \Cref{2sumcompositions}, namely, by expressing the left side of \eqref{estimatesumcompositions1} (which originally appears as a sum over three compositions) as a sum over one composition. To that end, for any nonnegative composition $D = (D_1, D_2, \ldots , D_r) \in \mathcal{G}_{U + V + W} (r)$, let $N_D (U, V, W)$ denote the number of triples $(A, B, C)$ of nonnegative compositions $A = (A_1, A_2, \ldots , A_r) \in \mathcal{G}_U (r)$; $B = (B_1, B_2, \ldots , B_r) \in \mathcal{G}_V (r)$; and $C = (C_1, C_2, \ldots , C_r) \in \mathcal{G}_W (r)$ such that $D_i = A_i + B_i + C_i$ for each $i \in [1, r]$. Then, 
	\begin{flalign*}
	\displaystyle\sum_{A \in \mathcal{G}_U (r)} \displaystyle\sum_{B \in \mathcal{G}_V (r)} \displaystyle\sum_{C \in \mathcal{G}_W (r)} \displaystyle\prod_{i = 1}^r (A_i + B_i + C_i + 1)! \le \displaystyle\sum_{D \in \mathcal{G}_{U + V + W} (r)} N_{D} (U, V, W) \displaystyle\prod_{i = 1}^r (D_i + 1)!.
	\end{flalign*}
	
	Next, observe as in the proof of \Cref{2sumcompositions} that any nonnegative composition $D \in \mathcal{G}_{U + V + W} (r)$ can be associated with a subset $\mathcal{I} \subset [1, r]$ indicating the locations of the entries equal to $0$ in $D$ and a (positive) composition $E = (E_1, E_2, \ldots , E_s) \in \mathcal{C}_{U + V + W} (s)$ denoting the nonzero elements of $D$ (here, $s = r - |\mathcal{I}|$). Since there are at most $2^r$ ways to select $\mathcal{I}$, it follows that 
	\begin{flalign}
	\label{sumcompositionsestimate}
	\begin{aligned}
	\displaystyle\sum_{A \in \mathcal{G}_U (r)} \displaystyle\sum_{B \in \mathcal{G}_V (r)} \displaystyle\sum_{C \in \mathcal{G}_W (r)} \displaystyle\prod_{i = 1}^r (A_i + B_i + C_i + 1)! \le 2^r \displaystyle\max_{s \in [1, r]} \displaystyle\sum_{E \in \mathcal{C}_{U + V + W} (s)} N_{E} (U, V, W) \displaystyle\prod_{i = 1}^r (E_i + 1)!.
	\end{aligned}
	\end{flalign}

	To analyze the sum on the right side of \eqref{sumcompositionsestimate}, we recall from the proof of \Cref{2sumcompositions} that $\mathcal{C}_{U + V + W} (s; m)$ denotes the set of compositions $E = (E_1, E_2, \ldots , E_s) \in \mathcal{C}_{U + V + W} (s)$ such that $\max_{i \in [1, s]} E_i \le m$.
	
	Let us first analyze the contribution of compositions $E \in \mathcal{C}_{U + V + W} (s) \setminus \mathcal{C}_{U + V + W} (s; U + V + W - s)$ to the sum on the right side of \eqref{sumcompositionsestimate}. To that end, observe that any such composition has one entry equal to $U + V + W - s + 1$ and the remaining entries equal to $1$; therefore, there are at most $r$ such compositions. 
	
	Furthermore, for any such composition $E$, we have that $N_E (U, V, W) \le 3^{r - 1} r^2$. Indeed, let $j \in [1, s]$ denote the index such that $E_j = U + V + W - s + 1$. Then, for any $A \in \mathcal{G}_U (s)$, $B \in \mathcal{G}_V (s)$, and $C \in \mathcal{G}_W (s)$ summing to $E$, there are at most $3$ possibilities for $(A_i, B_i, C_i)$ for each $i \ne j$ (namely, it can be $(1, 0, 0)$, $(0, 1, 0)$, or $(0, 0, 1)$) and at most $s^2$ possibilities for $(A_j, B_j, C_j)$ since we must have that $A_j \in [U - s + 1, U]$ and $B_j \in [V - s + 1, V]$ (and then $C_j = E_j - A_j - B_j$ would be determined).
	
	Thus, the contribution of compositions $E \in \mathcal{C}_{U + V + W} (s) \setminus \mathcal{C}_{U + V + W} (s; U + V + W - s)$ to the sum on the right side of \eqref{sumcompositionsestimate} is bounded by $r^3 3^{r - 1} (U + V + W - s + 2)!$. 
	
	Next, let us bound the contribution of $E \in \mathcal{C}_{U + V + W} (s, U + V + W - s) \setminus \mathcal{C}_{U + V + W} (s; U + V + W - s - 1)$ to the sum on the right side of \eqref{sumcompositionsestimate}. To that end, observe that any such composition has one entry equal to $U + V + W - s$, one entry equal to $2$, and the remaining entries equal to $1$; thus, there are at most $s (s - 1) \le r^2$ such compositions $E$. Furthermore, for any such $E$, we have that $\prod_{i = 1}^r (E_i + 1)! \le 2^{r - 1} 3 (U + V + W - s + 1)!$, in view of the second estimate in \eqref{aici2} (applied with the $k$ there equal to $r$ here; the $A_i$ there equal to the $E_i - 1$ here; and the $N$ there equal to $U + V + W - s$ here). 
	
	Additionally, for any such $E$, we have that $N_E (U, V, W) \le 3^{r - 2} 6 (r + 1)^2$. Indeed, let $j, k \in [1, s]$ denote the indices such that $E_j = U + V + W - 1$ and $E_k = 2$. Then, for any $A \in \mathcal{G}_U (s)$, $B \in \mathcal{G}_V (s)$, and $C \in \mathcal{G}_W (s)$ summing to $E$, there are at most $6$ ways to select $(A_k, B_k, C_k)$ summing to $2$; at most $(r + 1)^2$ ways to select $(A_j, B_j, C_j)$ (since $A_j \in [U - s, J]$, $B_j \in [V - s, V]$, and $C_j = E_j - A_j - B_j$); and at most $3^{r - 2}$ ways to select the remaining $(A_i, B_i, C_i)$. 
	
	Therefore, the total contribution of $E \in \mathcal{C}_{U + V + W} (s, U + V + W - s) \setminus \mathcal{C}_{U + V + W} (s; U + V + W - s - 1)$ to the sum on the right side of \eqref{sumcompositionsestimate} at most $6^{r + 1} r^4  (U + V + W - s + 1)!$. 
	
	Hence,
	\begin{flalign*}
	\displaystyle\sum_{E \in \mathcal{C}_{U + V + W} (s)} N_{E} (U, V, W) \displaystyle\prod_{i = 1}^r (E_i + 1)! & \le (3^{r - 1} r^3 + 6^{r + 1} r^4) (U + V + W - s + 2)! \\
	& \qquad + \displaystyle\sum_{E \in \mathcal{C}_{U + V + W} (s; U + V + W - s - 1)} N_E (U, V, W) \displaystyle\prod_{i = 1}^r (E_i + 1)!.
	\end{flalign*}
	
	Now observe that $N_E (U, V, W) \le \prod_{i = 1}^r (E_i + 1)^2$, due to the fact that if $A$, $B$, and $C$ are nonnegative compositions as above that sum to $E$, then there are at most $E_i + 1$ possibilities for each $A_i \in [0, E_i]$ and for each $B_i \in [0, E_i]$; then each $C_i$ is determined from $(A_i, B_i)$ and $E$. Thus, it follows that
	\begin{flalign*}
	\displaystyle\sum_{E \in \mathcal{C}_{U + V + W} (s)} N_{E} (U, V, W) \displaystyle\prod_{i = 1}^r (E_i + 1)! &  \le \displaystyle\sum_{E \in \mathcal{C}_{U + V + W} (s; U + V + W - s - 1)} \displaystyle\prod_{i = 1}^r (E_i + 3)! \\ 
	& \qquad + 6^{r + 2} r^4 (U + V + W - s + 2	)!.
	\end{flalign*}
	
	Applying \eqref{aici3} with the $k$ there equal to $r$ here, the $A_i$ there equal to the $E_i - 1$ here, and the $N$ there equal to $U + V + W - s$ here, we obtain that $\prod_{i = 1}^r (E_i + 3)! \le 4^{r - 2} 720 (U + V + W - s + 2)!$. Additionally, since $\big| \mathcal{C}_{U + V + W} (s) \big| \le \binom{U + V + W - 1}{s - 1}$ due to the first identity in \eqref{ynkestimate}, it follows that 
	\begin{flalign}
	\label{sumcompositionsestimate1} 
	\begin{aligned}
	& \displaystyle\sum_{E \in \mathcal{C}_{U + V + W} (s)} N_{E} (U, V, W) \displaystyle\prod_{i = 1}^r (E_i + 1)! \\
	& \qquad  \le 4^{r + 3} \binom{U + V + W - 1}{s - 1} (U + V + W - s + 2)! + 6^{r + 4} r^4 (U + V + W - s + 2)! \\
	& \qquad \le 2^{7r + 9} (U + V + W + 1)!,
	\end{aligned}
	\end{flalign}
	
	\noindent where we have used the first estimate in \eqref{2ll} and the fact that $s \ge 1$. 
	
	The lemma now follows from inserting \eqref{sumcompositionsestimate1} into \eqref{sumcompositionsestimate}. 
\end{proof}

\section{Asymptotics for the Saddle Connection Siegel--Veech Constant}

\label{ConstantAsymptotic} 

In this section we establish \Cref{sgestimate}. We begin in \Cref{IdentityConstant1} by providing an explicit, combinatorial formula for the Siegel--Veech constant $c_{\text{sc}}$, which is due to \cite{PBC}. Then, in \Cref{Asymptoticp1}, we use this formula, as well as the volume estimates provided in \Cref{EstimatesVolume} and some of the combinatorial estimates provided in \Cref{Estimates1}, to prove \Cref{sgestimate}. Throughout this section, we assume that $i = 1$ and $j = 2$ for notational convenience.

\subsection{A Combinatorial Formula for \texorpdfstring{$c_{\rm{sc}}^{m_1, m_2} \big( \mathcal{H} (\alpha) \big)$}{}} 

\label{IdentityConstant1}

In this section we state a combinatorial formula for the Siegel--Veech constant $c_{\text{sc}}^{m_1, m_2} \big( \mathcal{H} (\alpha) \big)$, which is originally due to \cite{PBC}. Next, we state an estimate on related constants, given by \Cref{cpm1m2estimate}, which will be established in \Cref{Asymptoticp1} below. Assuming \Cref{cpm1m2estimate}, we then establish \Cref{sgestimate}. 

In order to state the combinatorial formula for the Siegel--Veech constant $c_{\text{sc}}^{m_1, m_2} \big( \mathcal{H} (\alpha) \big)$, we begin with some notation. For any positive integers $p$ and $k$ and any $k$-tuple $\Delta = (d_1, d_2, \ldots , d_k)$ of positive integers, let $\mathfrak{N}_p (\Delta)$ denote the family of ordered set partitions of $\Delta$, that is, the family of ordered $p$-tuples of mutually disjoint (possibly empty) subsets $( \Delta_1, \Delta_2, \ldots , \Delta_p)$ of $\Delta$ such that $\bigcup_{j = 1}^p \Delta_j = \Delta$. 

Now fix an $n$-tuple $\alpha = (m_1, m_2, \ldots , m_n)$ of positive integers; let $\overline{\alpha} = (m_3, m_4, \ldots , m_n)$ denote the $(n - 2)$-tuple obtained by removing the first two elements of $\alpha$. Recall that, if $|\alpha|$ is odd, we define $\nu_1 \big( \mathcal{H} (\alpha) \big) = 0$. 

Under the above notation, we have the following definition.

\begin{definition}
	
	Define
	\begin{flalign}
	\label{cm1m2p} 
	\begin{aligned}
	c_{m_1, m_2}^{(p)} (\alpha) =  \displaystyle\frac{1}{\nu_1 \big( \mathcal{H}(\alpha) \big)} \displaystyle\frac{2^{1 - p}}{\big( |\alpha| + n - 1 \big)!} &   \displaystyle\sum_{( \overline{\alpha}_i ) \in \mathfrak{N}_p (\overline{\alpha})} \displaystyle\sum_{\substack{a' \in \mathcal{G}_p (m_1 - p + 1) \\ a'' \in \mathcal{G}_p (m_2 - p + 1)}} \displaystyle\prod_{i = 1}^p  \big( |\overline{\alpha}_i| + a_i' + a_i'' + \ell (\overline{\alpha}_i) + 1 \big)! \\
	& \qquad \qquad \qquad \times (a_i' + a_i'' + 1)\nu_1 \Big( \mathcal{H} \big( \overline{\alpha}_i \cup \{ a_i' + a_i'' \} \big) \Big),
	\end{aligned}
	\end{flalign}

	\noindent where $( \overline{\alpha}_1, \overline{\alpha}_2, \ldots , \overline{\alpha}_p) \in \mathfrak{N}_p (\overline{\alpha})$ denotes a set partition; $a' = (a_1', a_2', \ldots , a_p') \in \mathcal{G}_p (m_1 - p + 1)$ and $a'' = (a_1'', a_2'', \ldots , a_p'') \in \mathcal{G}_p (m_2 - p + 1)$ denote nonnegative compositions ; and $\ell (\overline{\alpha}_i)$ denotes the number of elements in $\overline{\alpha}_i$ for each $i \in [1, p]$.
	
\end{definition}

\begin{rem}

For any positive integer $p$, the quantity $c_{m_1, m_2}^{(p)} \big( \mathcal{H} (\alpha) \big)$ denotes the contribution to the Siegel--Veech constant $c_{\text{sc}}^{m_1, m_2} \big(\mathcal{H} (\alpha) \big)$ that arises from counting configurations of $p$ homologous saddle connections joining two distinct, fixed zeros of orders $m_1$ and $m_2$ on a generic flat surface in $\mathcal{H}(\alpha)$. 

\end{rem}

Now the following proposition due to Lemma 9.1 and Section 9.5 of \cite{PBC} provides an explicit formula for the Siegel--Veech constant $c_{\text{sc}}^{m_1, m_2} \big( \mathcal{H} (\alpha) \big)$.

\begin{prop}[{\cite[Lemma 9.1 and Section 9.5]{PBC}}]

Assume that the stratum $\mathcal{H} (\alpha)$ is nonempty and connected. Then, under the above notation, we have that 
\begin{flalign}
\label{sm1m2malpha}
c_{\text{\emph{sc}}}^{m_1, m_2} \big( \mathcal{H} (\alpha) \big) = \displaystyle\sum_{p = 1}^{\min \{ m_1, m_2 \} + 1} c_{m_1, m_2}^{(p)} (\alpha). 
\end{flalign}

\end{prop}

In \Cref{Asymptoticp1} we will establish the following bound on these Siegel--Veech constants for $p \ge 2$. 

\begin{prop} 
	
\label{cpm1m2estimate} 

There exists a constant $C > 0$ such that the following holds. For any integer $g \ge 3$; partition $\alpha \in \mathbb{Y}_{2g - 2}$ such the stratum $\mathcal{H} (\alpha)$ is nonempty and connected; and integer $p \ge 2$, we have that 
\begin{flalign}
\label{m1m2estimatecm1m2palpha} 
\big| c_{m_1, m_2}^{(p)} (\alpha) \big| \le \displaystyle\frac{(m_1 + 1) (m_2 + 1)}{g} \displaystyle\frac{C^p}{|\alpha|^{2p - 3}}. 
\end{flalign}
\end{prop}

Assuming \Cref{cpm1m2estimate}, we can now establish \Cref{sgestimate}.

\begin{proof}[Proof of \Cref{sgestimate} Assuming \Cref{cpm1m2estimate}]
	
	We may assume that $g$ is sufficiently large. Thus, throughout this proof, we will denote the constant $R = 2^{2^{200}}$ and assume that $g > R^2$. 
	
	We begin by approximating $c_{m_1, m_2}^{(p)}$ using \eqref{cm1m2p} and \eqref{volumeestimate1h1}. Such a bound (see \eqref{nu1alphanu1alpha} below) was in fact originally established as Corollary 1 in the Appendix of \cite{LGAVSD}; we essentially repeat that proof below. 
	
	Applying \eqref{cm1m2p} with $p = 1$ yields 
	\begin{flalign*}
	c_{m_1, m_2}^{(1)} (\alpha)  = \displaystyle\frac{\big( |\overline{\alpha}| + m_1 + m_2 + \ell ( \overline{\alpha}) + 1 \big)!}{\big( |\alpha| + n - 1 \big)!} \displaystyle\frac{(m_1 + m_2 + 1) \nu_1 \big( \mathcal{H} (\overline{\alpha} \cup \{ m_1 + m_2 \}) \big)}{ \nu_1 \big( \mathcal{H} (\alpha) \big)},
	\end{flalign*}
	
	\noindent where we have used the fact that the $\overline{\alpha}_1$ in \eqref{cm1m2p} is equal to $\overline{\alpha}$; that the $a'$ there is equal to $(m_1)$; and that the $a_2$ there is equal to $(m_2)$. Since $|\overline{\alpha}| + m_1 + m_2 = |\alpha|$ and $\ell (\overline{\alpha}) = \ell (\alpha) - 2 = n - 2$, it follows that 
	\begin{flalign}
	\label{nu1alphanu1alpha} 
	\displaystyle\frac{c_{m_1, m_2}^{(1)} (\alpha)}{(m_1 + 1) (m_2 + 1)} = \displaystyle\frac{(m_1 + m_2 + 1) \nu_1 \big( \mathcal{H} (\overline{\alpha} \cup \{ m_1 + m_2 \}) \big)}{(m_1 + 1) (m_2 + 1) \nu_1 \big( \mathcal{H} (\alpha) \big)}.
	\end{flalign}
	
	\noindent Now \eqref{volumeestimate1h1} implies that 
	\begin{flalign}
	\label{estimatenu1alphan1alpha}
	\begin{aligned}
	& 4 \left( 1 - \displaystyle\frac{R}{g}\right)  \le (m_1 + m_2 + 1) \nu_1 \big( \mathcal{H} (\overline{\alpha} \cup \{ m_1 + m_2 \}) \big) \displaystyle\prod_{i = 3}^n (m_i + 1) \le 4 \left( 1 + \displaystyle\frac{R}{g}\right); \\
	& 4 \left( 1 - \displaystyle\frac{R}{g}\right)  \le (m_1 + 1) (m_2 + 1)  \nu_1 \big( \mathcal{H} (\overline{\alpha} \cup \{ m_1 + m_2 \}) \big) \displaystyle\prod_{i = 3}^n (m_i + 1) \le 4 \left( 1 + \displaystyle\frac{R}{g}\right).
	\end{aligned}
	\end{flalign}
		
	\noindent Inserting \eqref{estimatenu1alphan1alpha} into \eqref{nu1alphanu1alpha} therefore yields 
	\begin{flalign}
	\label{2nu1alphanu1alpha} 
	\left| \displaystyle\frac{c_{m_1, m_2}^{(1)} (\alpha)}{(m_1 + 1) (m_2 + 1)} - 1 \right| \le 1 + \displaystyle\frac{3 R}{g}. 
	\end{flalign}
	
	\noindent Now, applying \eqref{sm1m2malpha}, \eqref{m1m2estimatecm1m2palpha}, and \eqref{2nu1alphanu1alpha}, we find that 
	\begin{flalign}
	\label{sm1m2alphaestimate1}
	\begin{aligned}
	\big| c_{\text{sc}}^{m_1, m_2} & \big( \mathcal{H} (\alpha) \big) - (m_1 + 1) (m_2 + 1) \big| \\
	& \le \big| c_{m_1, m_2}^{(1)} (\alpha) - (m_1 + 1) (m_2 + 1) \big|  + \displaystyle\sum_{p = 2}^{\min \{ m_1, m_2 \} + 1} c_{m_1, m_2}^{(p)} (\alpha) \\
	& \le \displaystyle\frac{3 R (m_1 + 1) (m_2 + 1)}{g} + \displaystyle\frac{(m_1 + 1)(m_2 + 1)}{g} \displaystyle\sum_{p = 2}^{\min \{ m_1, m_2 \} + 1} \displaystyle\frac{C^p}{|\alpha|^{2p - 3}} \\
	& \le \displaystyle\frac{C' (m_1 + 1) (m_2 + 1)}{g},
	\end{aligned}
	\end{flalign}
	
	\noindent for some sufficiently large $C' > 0$, where here we have used the fact that $|\alpha| \ge \max \{ g, m_1, m_2 \} + 1$. The theorem now follows from \eqref{sm1m2alphaestimate1}. 
	\end{proof}

\subsection{Proof of Proposition \ref{cpm1m2estimate}}

\label{Asymptoticp1}

Here, we prove Proposition \ref{cpm1m2estimate}. Throughout this section, we abbreviate $c_{m_1, m_2}^{(p)} = c_{m_1, m_2}^{(p)} (\alpha)$.

\begin{proof}[Proof of Proposition \ref{cpm1m2estimate}]
	
First observe that, by \eqref{volumeestimate}, there exists a sufficiently large constant $R > 2$ such that 
\begin{flalign}
\label{productnuestimate}
 \displaystyle\frac{1}{\nu_1 \big( \mathcal{H} (\alpha) \big)} \displaystyle\prod_{i = 1}^p \nu_1 \Big( \mathcal{H} \big( \overline{\alpha}_i \cup \{ a_i' + a_i'' \}  \big) \Big) (a_i' + a_i'' + 1) \le R^{p + 1} (m_1 + 1) (m_2 + 1),
\end{flalign}

\noindent for any $\big(\overline{\alpha}_1, \overline{\alpha}_2, \ldots , \overline{\alpha}_p \big) \in \mathfrak{N}_p (\overline{\alpha})$. Inserting \eqref{productnuestimate} into \eqref{cm1m2p}, we obtain 
\begin{flalign}
\begin{aligned}
\label{testimate1}
c_{m_1, m_2}^{(p)}  & \le (m_1 + 1)(m_2 + 1) \displaystyle\frac{R^{p + 1}}{2^{p - 1} \big( |\alpha| + n - 1 \big)!} \\
& \qquad  \times  \displaystyle\sum_{ (\overline{\alpha}_i ) \in \mathfrak{N}_p (\overline{\alpha})} \displaystyle\sum_{a' \in \mathcal{G}_p (m_1 + 1 - p)} \displaystyle\sum_{ a'' \in \mathcal{G}_p (m_2 + 1 - p)} \displaystyle\prod_{i = 1}^p \big( |\overline{\alpha}_i| + a_i' + a_i'' + \ell (\overline{\alpha}_i) + 1 \big)!.
\end{aligned}
\end{flalign}

In order to bound the right side of \eqref{testimate1}, let us decompose the set $\mathfrak{N}_p (\overline{\alpha})$ as follows. For any nonnegative composition $C = (C_1, C_2, \ldots , C_p) \in \mathcal{G}_{n - 2} (p)$, let $\mathfrak{N} (\overline{\alpha}; C)$ denote the family of set partitions $(\overline{\alpha}_1, \overline{\alpha}_2, \ldots , \overline{\alpha}_p)$ of $\overline{\alpha}$ such that $\ell (\overline{\alpha}_i) = C_i$ for each $i \in [1, p]$. Thus, $\mathfrak{N} (\overline{\alpha}; C)$ denotes the family of set partitions of $\overline{\alpha}$ into $p$ components of lengths $C_1, C_2, \ldots , C_p$ (in this order). 

Since $\mathfrak{N} (\overline{\alpha}) = \bigcup_{C \in \mathcal{G}_{n - 2} (p)} \mathfrak{N} (\overline{\alpha}; C)$, we deduce from \eqref{testimate1} that
\begin{flalign} 
\label{s2}
\begin{aligned}
c_{m_1, m_2}^{(p)} & \le (m_1 + 1) (m_2 + 1) \displaystyle\frac{R^{p + 1}}{\big( |\alpha| + n - 1 \big)!} \\
& \qquad \times \displaystyle\sum_{C \in \mathcal{G}_{n - 2} (p)}  \displaystyle\sum_{( \overline{\alpha}_i ) \in \mathfrak{N} (\overline{\alpha}; C)} \displaystyle\sum_{a' \in \mathcal{G}_{m_1 - p + 1} (p)} \displaystyle\sum_{ a'' \in \mathcal{G}_{m_2 - p + 1} (p)} \displaystyle\prod_{i = 1}^p \big( |\overline{\alpha}_i| + a_i' + a_i'' + C_i + 1 \big)!.
\end{aligned}
\end{flalign}

Next, we will estimate the maximum possible value of the product on the right side of \eqref{s2} over all $( \overline{\alpha}_i ) \in \mathfrak{N} (\overline{\alpha}; C)$ for some fixed $C = (C_1, C_2, \ldots,  C_p) \in \mathcal{G}_{n - 2} (p)$. To that end, observe that since all entries of $\alpha$ are positive, $|\overline{\alpha}_i| \ge \ell (\overline{\alpha}_i) = C_i$, meaning that $|\overline{\alpha}_i| + C_i \ge 2 C_i$ for each $i \in [1, p]$. 

Now let us apply \Cref{aibiproduct} with the $A_i$ there equal to the $|\overline{\alpha}_i|$ here; the $B_i$ there equal to the $a_i' + a_i'' + C_i$ here; the $C_i$ there equal to the $C_i$ here; and the $d$ there equal to $1$ here. Since
\begin{flalign*} 
\displaystyle\sum_{i = 1}^p  |\overline{\alpha}_i| = |\alpha| - m_1 - m_2; \qquad \displaystyle\sum_{i = 1}^p (a_i' + a_i'') = m_1 + m_2 - 2p + 2; \qquad \displaystyle\sum_{i = 1}^p \ell (\overline{\alpha}_i) =  n - 2,  
\end{flalign*} 

\noindent we have that the $U$ there is equal to $|\alpha| - m_1 - m_2$ here; the $V$ there is equal to $n + m_1 + m_2 - 2p$ here; and the $W$ there is equal to $n - 2$ here. Thus, \Cref{aibiproduct} implies that
  \begin{flalign*} 
 \displaystyle\prod_{i = 1}^p \big( |\overline{\alpha}_i| + a_i' + a_i'' + \ell (\overline{\alpha}_i) + 1 \big)!  \le \displaystyle\frac{\big( |\alpha| + n - 2p + 1  \big)!}{(2n + m_1 + m_2 - 2p - 1)!}  \displaystyle\prod_{i = 1}^p (2C_i + a_i' + a_i'' + 1)!,
 \end{flalign*}
 
 \noindent which upon insertion into \eqref{s2} yields 
 \begin{flalign*}
 c_{m_1, m_2}^{(p)} & \le \displaystyle\frac{(m_1 + 1)(m_2 + 1)  }{\big( |\alpha| + n - 1 \big)! } \displaystyle\frac{R^{p + 1} \big( |\alpha| + n - 2p + 1 \big)!}{(2n + m_1 + m_2 - 2p - 1)!}  \\
 & \qquad \sum_{C \in \mathcal{G}_{n - 2} (p)} \displaystyle\sum_{( \overline{\alpha}_i) \in \mathfrak{N} (\overline{\alpha}; C)} \displaystyle\sum_{a' \in \mathcal{G}_{m_1 - p + 1} (p)} \displaystyle\sum_{ a'' \in \mathcal{G}_{m_2 - p + 1} (p)} \displaystyle\prod_{i = 1}^p (2C_i + a_i' + a_i'' + 1)! \\
 & = \displaystyle\frac{(m_1 + 1)(m_2 + 1)  }{\big( |\alpha| + n - 1 \big)! }  \displaystyle\sum_{p = 2}^{m_1 + m_2} \displaystyle\frac{R^{p + 1} \big( |\alpha| + n - 2p + 1 \big)!}{(2n + m_1 + m_2 - 2p - 1)!}  \\
 & \qquad  \times \displaystyle\sum_{C \in \mathcal{G}_{n - 2} (p)} \binom{n - 2}{C_1, C_2, \ldots , C_p} \displaystyle\sum_{a' \in \mathcal{G}_{m_1 - p + 1} (p)} \displaystyle\sum_{ a'' \in \mathcal{G}_{m_2 - p + 1} (p)}  \displaystyle\prod_{i = 1}^p (2C_i + a_i' + a_i'' + 1)!,
 \end{flalign*}

\noindent where to deduce the equality we used the fact that $\big| \mathfrak{N} (\overline{\alpha}; C) \big| \le \binom{n}{C_1, C_2, \ldots , C_p}$ (due to the first identity in \eqref{aipksize}). Thus, since 
\begin{flalign*} 
\displaystyle\frac{(2 C_i + a_i' + a_i'' + 1)!}{C_i!} & = \binom{2 C_i + a_i' + a_i'' + 1}{C_i} (C_i + a_i' + a_i'' + 1)! \\
 & \le 2 \binom{2 C_i + a_i' + a_i''}{C_i} (C_i + a_i' + a_i'' + 1)!,
\end{flalign*} 

\noindent it follows that
 \begin{flalign} 
\label{s3}
\begin{aligned}
c_{m_1, m_2}^{(p)} & \le  \displaystyle\frac{(m_1 + 1)(m_2 + 1) (n - 2)!}{\big( |\alpha| + n - 1 \big)!} \displaystyle\frac{(2 R)^{p + 1} \big( |\alpha| + n - 2p + 1 \big)!}{(2n + m_1 + m_2 - 2p - 1)!}  \\
& \qquad  \times \displaystyle\sum_{C \in \mathcal{G}_{n - 2} (p)} \displaystyle\sum_{a' \in \mathcal{G}_{m_1 - p + 1} (p)} \displaystyle\sum_{ a'' \in \mathcal{G}_{m_2 - p + 1} (p)} \displaystyle\prod_{i = 1}^p \binom{2C_i + a_i' + a_i''}{C_i} (C_i + a_i' + a_i'' + 1)!.
\end{aligned}
\end{flalign}

Next we apply \eqref{aijai}, with the $r$ there equal to $2$; the $A_{i, 1}$ there equal to the $C_i$ here; and the $A_{i, 2}$ there equal to the $C_i + a_i' + a_i''$ here. Since $\sum_{i = 1}^p C_i = n - 2$ and $\sum_{i = 1}^p (a_i' + a_i'') = m_1 + m_2 - 2p + 2$, we deduce that
\begin{flalign*}
\displaystyle\prod_{i = 1}^p \binom{2C_i + a_i' + a_i''}{C_i} \le \binom{2n + m_1 + m_2 - 2p - 2}{n - 2}!, 
\end{flalign*}

\noindent which upon insertion into \eqref{s3} yields
 \begin{flalign*}
c_{m_1, m_2}^{(p)} & \le \displaystyle\frac{(m_1 + 1)(m_2 + 1) (n - 2)!}{\big( |\alpha| + n - 1 \big)!}   \displaystyle\frac{(2 R)^{p + 1} \big( |\alpha| + n - 2p + 1 \big)! }{(2n + m_1 + m_2 - 2p - 1)!} \binom{2n + m_1 + m_2 - 2p - 2}{n - 2}	  \\
& \qquad \displaystyle\sum_{C \in \mathcal{G}_{n - 2} (p)}  \displaystyle\sum_{a' \in \mathcal{G}_{m_1 - p + 1} (p)} \displaystyle\sum_{ a'' \in \mathcal{G}_{m_2 - p + 1} (p)} \displaystyle\prod_{i = 1}^p (C_i + a_i' + a_i'' + 1)! \\
& = \displaystyle\frac{(m_1 + 1)(m_2 + 1) }{\big( |\alpha| + n - 1 \big)!}  \displaystyle\sum_{p = 2}^{m_1 + m_2} \displaystyle\frac{(2 R)^{p + 1} \big( |\alpha| + n - 2p + 1 \big)! }{(2n + m_1 + m_2 - 2p - 1)!} \displaystyle\frac{(2n + m_1 + m_2 - 2p - 2)!}{(n + m_1 + m_2 - 2p)!}	 	  \\
& \qquad \times \displaystyle\sum_{C \in \mathcal{G}_{n - 2} (p)}  \displaystyle\sum_{a' \in \mathcal{G}_{m_1 - p + 1} (p)} \displaystyle\sum_{ a'' \in \mathcal{G}_{m_2 - p + 1} (p)} \displaystyle\prod_{i = 1}^p (C_i + a_i' + a_i'' + 1)!.
\end{flalign*}

\noindent Since $|\alpha| = 2g - 2 \ge g$ and since 
\begin{flalign*}
\displaystyle\frac{(|\alpha| + n  - 2p + 1)!}{(|\alpha| + n - 1)!} \le \displaystyle\frac{2^{4p - 3}}{\big( |\alpha| + n + 1 \big)^{2p - 2}} \le \displaystyle\frac{2^{4p - 3}}{g |\alpha|^{2p - 3}},
\end{flalign*} 

\noindent due to the second estimate in \eqref{2ll}, we obtain that  
\begin{flalign}
\label{s4}
\begin{aligned} 
c_{m_1, m_2}^{(p)} & \le \displaystyle\frac{(m_1 + 1)(m_2 + 1)}{g} \displaystyle\frac{(32 R)^{p + 1}}{|\alpha|^{2p - 3} (n + m_1 + m_2 - 2p + 1)!}  \\
& \qquad \displaystyle\sum_{C \in \mathcal{G}_{n - 2} (p)}  \displaystyle\sum_{a' \in \mathcal{G}_{m_1 - p + 1} (p)} \displaystyle\sum_{ a'' \in \mathcal{G}_{m_2 - p + 1} (p)} \displaystyle\prod_{i = 1}^p (C_i + a_i' + a_i'' + 1)!.
\end{aligned}
\end{flalign}

 Now we apply Lemma \ref{sumcompositions} to \eqref{s4}, with the $r$ there equal to $p$ here; the $A_i$ there equal to the $a_i'$ here; the $B_i$ there equal to the $a_i''$ here; the $C_i$ there equal to the $C_i$ here; the $U$ there equal to $m_1 - p + 1$ here; the $V$ there equal to $m_2 - p + 1$ here; and the $W$ there equal to $n - 2$ here. This yields 
 \begin{flalign*}
c_{m_1, m_2}^{(p)} & \le   \displaystyle\frac{(m_1 + 1)(m_2 + 1)}{g} \displaystyle\frac{(2^{22} R)^{p + 1}}{|\alpha|^{2p - 3}},
\end{flalign*}

\noindent from which we deduce the proposition. 
\end{proof}

\section{Asymptotics for the Area Siegel--Veech Constant} 

\label{AreaConstant}

In this section we establish \Cref{constantareaasymptotic}. We begin in \Cref{IdentityConstantArea} by providing a combinatorial formula for the area Siegel--Veech constant $c_{\text{area}} \big(\mathcal{H} (\alpha) \big)$ due to \cite{ERG,PBC,PGTS}. Next, in \Cref{Termp1}, we analyze the leading order term of this expression, reducing \Cref{constantareaasymptotic} to \Cref{xareaestimate} below. We then establish \Cref{xareaestimate} in \Cref{Termp2}.  

Throughout this section, we let $g > 2$ be a positive integer and $\alpha = (m_1, m_2, \ldots , m_n) \in \mathbb{Y}_{2g - 2}$ be a partition such that the stratum $\mathcal{H} (\alpha)$ is nonempty and connected.

\subsection{A Combinatorial Formula for \texorpdfstring{$c_{\rm{area}} \big(\mathcal{H} (\alpha)\big)$}{}} 

\label{IdentityConstantArea} 

In order to provide a combinatorial formula for the area Siegel--Veech constant $c_{\text{area}} \big(\mathcal{H} (\alpha) \big)$, let us first fix positive integers $p \le g$ and $d \le \max \{ 2p, n \}$. 

Now let $\mathfrak{M}_d (\alpha)$ denote the family of $d$-tuples $\mu = (\mu_1, \mu_2, \ldots , \mu_d) \subseteq (m_1, m_2, \ldots , m_n)$; we assume here that $\mu$ is ordered, meaning that we distinguish between two $d$-tuples $\mu$ consisting of the same $\mu_i$ but in a different order. Observe in particular that $\big| \mathfrak{M}_d (\alpha) \big| = \frac{n!}{(n - d)!}$. Let $\overline{\alpha} = \alpha \setminus \mu$, which consists of $n - d$ elements. 

Recalling the notation from \Cref{IdentityConstant1}, let $\mathfrak{N}_p (\overline{\alpha})$ denote the family of set partitions of $\overline{\alpha}$, that is, the family of $p$-tuples of mutually disjoint (possibly empty) subsets $\big( \overline{\alpha}_1, \overline{\alpha}_2, \ldots , \overline{\alpha}_p \big)$ such that $\bigcup_{i = 1}^p \overline{\alpha}_i = \overline{\alpha}$. Here, we assume that this set partition is ordered, meaning that we distinguish between two $p$-tuples consisting of the same $\overline{\alpha}_i$ but in a different order; however, we assume that each individual $\overline{\alpha}_i$ is unordered, meaning that two components $\overline{\alpha}_i$ comprising the same elements but in a different order are not distinguished. For instance, if $\overline{\alpha} = \{ 1, 2, 3, 4 \}$, then we view the set partitions $\big( \{ 1, 2 \}, \{ 3, 4 \}\big)$ and $\big( \{ 3, 4 \}, \{ 1, 2\} \big)$ as different, but $\big( \{ 1, 2 \}, \{ 3, 4 \} \big)$ and $\big( \{ 2, 1 \}, \{ 3, 4 \} \big)$ as equivalent. 

Next, let $s = (s_1, s_2, \ldots , s_d) \in \mathcal{C}_{2p} (d)$ denote a (positive) composition of $2p$ of length $d$; the fact that $d \le 2p$ guarantees that $\mathcal{C}_{2p} (d)$ is nonempty. Further let $\kappa^{(i)} = (\kappa_{i, 1}, \kappa_{i, 2}, \ldots , \kappa_{i, s_i}) \in \mathcal{C}_{\mu_i} (s_i)$ denote a (positive) composition of $\mu_i$ of length $s_i$ for each $i \in [1, p]$, and define the $2p$-tuple 
\begin{flalign*}
\mathcal{K} = ( \kappa_{1, 1}, \kappa_{1, 2} , \ldots , \kappa_{1, s_1}, \kappa_{2, 1}, \kappa_{2, 2}, \ldots , \kappa_{2, s_2}, \ldots , \kappa_{d, 1}, \kappa_{d, 2}, \ldots , \kappa_{d, s_d}),
\end{flalign*}

\noindent obtained by concatenating the compositions $\kappa^{(1)}, \kappa^{(2)}, \ldots , \kappa^{(d)}$ in this order.

For any integer $t \in [0, 2p - 1]$, let $\mathcal{K}_t$ denote the $2p$-tuple obtained by shifting each element of $\mathcal{K}$ to the left $t$ times; for instance, $\mathcal{K}_0 = \mathcal{K}$, and $\mathcal{K}_1$ ends with $\kappa_{1, 1}$ and starts with $\kappa_{1, 2}$ (or $\kappa_{2, 1}$ if $s_1 = 1$). Relabel the elements of $\mathcal{K}_t$ by denoting $\mathcal{K}_t = (c_1, c_2, \ldots , c_{2p})$, where the $c_i = c_i (t)$ constitute a suitable reordering of the $\kappa_{i, j}$ that depends on $t$. 

In particular, for each integer $i \in [1, d]$, there exists some $u_i \in [1, 2p]$ such that $c_{u_i} = \kappa_{i, 1}$. Let $\mathcal{J} = (j_1, j_2, \ldots , j_d)$ denote the reordering of the $\{ u_i \}$ such that $1 \le j_1 < j_2 < \cdots < j_d \le 2p$. Further let $q$ denote the number of odd elements in $\mathcal{J}$ (or, equivalently, the number of odd $u_i$). 

Next, for each integer $i \in [1, p]$ define the set $\alpha_i'$ as follows. If $2i \in \mathcal{J}$, then set $\alpha_i' = \overline{\alpha}_i \cup \{ c_{2i - 1} - 1, c_{2i} - 1 \}$. Otherwise, if $2i \notin \mathcal{J}$ then set $\alpha_i' = \overline{\alpha}_i \cup \{ c_{2i - 1} + c_{2i} - 2 \}$. 

Under the above notation, we have the following definition (where, in the below, we recall that we set $\nu_1 \big( \mathcal{H} (\alpha) \big) = 0$ whenever $|\alpha|$ is odd). 

\begin{definition}
	
\label{apd}
	
For all positive integers $p \le g$ and $d \le \max \{ n, 2p \}$, define 
\begin{flalign}
\label{apdalpha} 
\begin{aligned}
\mathcal{A}^{(p; d)} (\alpha) & = \displaystyle\frac{1}{\big(|\alpha| + n \big)!} \displaystyle\frac{1}{\nu_1 \big(\mathcal{H} (\alpha) \big)} \displaystyle\frac{1}{p d 2^p} \displaystyle\sum_{\mu \in \mathfrak{M}_d (n)} \displaystyle\sum_{(\overline{\alpha}_i) \in \mathfrak{N}_p (\overline{\alpha})} \displaystyle\sum_{s \in \mathcal{C}_{2p} (d)} \displaystyle\sum_{\kappa^{(i)} \in \mathcal{C}_{\mu_i} (s_i)} \displaystyle\sum_{t = 0}^{2p - 1} q \\
& \qquad \qquad \qquad  \times \displaystyle\prod_{i = 1}^p \big( |\alpha_i'| + \ell (\alpha_i') \big)! \nu_1 \big( \mathcal{H} (\alpha_i') \big) \displaystyle\prod_{\substack{1 \le i \le p \\ 2i \in \mathcal{J}}} c_{2i - 1} c_{2i}  \displaystyle\prod_{\substack{1 \le i \le p \\ 2i \notin \mathcal{J}}} (c_{2i - 1} + c_{2i} - 1).
 \end{aligned}
\end{flalign}

\end{definition} 

Now, we have the following proposition, obtained by combining equation (32) of \cite{PBC} with either \cite{PGTS} or equation (2.16) of \cite{ERG}.

\begin{prop}[{\cite[Equation (2.16)]{ERG}, \cite[Equation (32)]{PBC}, \cite{PGTS}}]

\label{areaalpha} 

Assume that the stratum $\mathcal{H} (\alpha)$ is nonempty and connected. Then, under the above notation, we have that 
\begin{flalign}
\label{capd}
c_{\text{\emph{area}}} \big(\mathcal{H} (\alpha) \big) =   \displaystyle\sum_{p = 1}^g \displaystyle\sum_{d = 1}^{\min \{ n, 2p \}}  \mathcal{A}^{(p; d)} (\alpha) .
\end{flalign}

\end{prop}

\subsection{Contribution of the \texorpdfstring{$p = 1$}{} Term} 

\label{Termp1}

In this section we analyze the contribution of the $p = 1$ terms to the area Siegel--Veech constant $c \big(\mathcal{H} (\alpha) \big)$. Next, we state an estimate for the remaining $\mathcal{A}^{p; d} (\alpha)$ given by \Cref{xareaestimate}, which we will establish in \Cref{Termp2} below. Then, assuming this estimate, we prove \Cref{constantareaasymptotic}. 

To implement the first task, let us define 
\begin{flalign}
\label{t} 
T (\alpha)  = \displaystyle\sum_{d = 1}^{\min \{ n, 2 \}}  \mathcal{A}^{(1; d)} (\alpha) =  \mathcal{A}^{(1; 1)} (\alpha) +  \textbf{1}_{n \ge 2} \mathcal{A}^{(1; 2)} (\alpha), 
\end{flalign}

\noindent which is the $p = 1$ contribution to the right side of \eqref{capd}. 

The following result, which approximates $T (\alpha)$ by $\frac{1}{2}$, appears as Theorem 1 in the appendix of \cite{LGAVSD}; we essentially repeat that proof below. 

\begin{prop}[{\cite[Theorem 1]{LGAVSD}}]
	
	\label{ap1estimate} 
	
	There exists a constant $C > 0$ such that the following holds. For any integer $g \ge 3$ and partition $\alpha = (m_1, m_2, \ldots , m_n) \in \mathbb{Y}_{2g - 2}$ such that the stratum $\mathcal{H} (\alpha)$ is nonempty and connected, we have that
	\begin{flalign*}
	\left| T(\alpha) - \displaystyle\frac{1}{2} \right| \le \displaystyle\frac{C}{g}. 
	\end{flalign*}

\end{prop} 

\begin{proof}

Throughout this proof, we abbreviate $T = T(\alpha)$. We also define the constant $R = 2^{2^{200}}$ and assume that $g > R^2$. 

We can express $T = T_1 + T_2$, where $T_1 = T_1 (\alpha) = \mathcal{A}^{1; 1} (\alpha)$ and $T_2 = T_2 (\alpha) = \textbf{1}_{n \ge 2} \mathcal{A}^{1; 2} (\alpha)$. To evaluate these two quantities explicitly, we apply \eqref{apdalpha}. 

Specifically, if the $(p, d)$ there is equal to $(1, 1)$, then $\mu = (\mu_1)$ can be equal to any of the $m_i \in \alpha$, in which case $\overline{\alpha}_1 = \alpha \setminus \{ m_i \}$. Then, $\mathcal{C}_{2p} (d) = \mathcal{C}_2 (1)$ consists of one element, and so $s = (2)$. Thus, $\kappa = (\kappa_{1, 1}, \kappa_{1, 2})$ ranges over all pairs of positive integers that sum to $m_i$. 

The value of $t$ can either be $0$ or $1$. If it is equal to $0$, then $c_1 = \kappa_{1, 1}$ and $c_2 = \kappa_{1, 2}$, so that $\mathcal{J} = (1)$ and $q = 1$; if it is equal to $1$, then $c_1 = \kappa_{1, 2}$ and $c_2 = \kappa_{1, 1}$, so that $\mathcal{J} = (2)$ and $q = 0$. Only the former case ($t = 0$; $c_1 = \kappa_{1, 1}$, $c_2 = \kappa_{1, 2}$, and $\mathcal{J} = (1)$) yields a nonzero contribution to the right side of \eqref{apdalpha}. Then, $\alpha_1' = \overline{\alpha} \cup \{ m_i - 2 \}$ (since $c_1 + c_2 = m_i$), meaning that $|\alpha_1'| = |\alpha| - 2$ and $\ell (\alpha_1') = \ell (\alpha) = n$. 

Inserting these facts into \eqref{apdalpha}, we deduce that  
\begin{flalign}
\label{t1} 
\begin{aligned}
T_1 & = \displaystyle\frac{1}{2}  \displaystyle\frac{1}{\big(|\alpha| + n \big)!} \displaystyle\frac{1}{\nu_1 \big(\mathcal{H} (\alpha) \big)} \displaystyle\sum_{i = 1}^n \displaystyle\sum_{c_1 + c_2 = m_i} (m_i - 1) \big( |\alpha| + n - 2 \big)! \nu_1 \big( \mathcal{H} (\alpha_i') \big) \\
& = \displaystyle\frac{1}{2} \displaystyle\frac{1}{\big(|\alpha| + n \big) \big( |\alpha| + n - 1\big) }  \displaystyle\sum_{i = 1}^n (m_i^2 - 1)  \displaystyle\frac{(m_i - 1) \nu_1 \big( \mathcal{H} (\alpha_i') \big)}{(m_i + 1) \nu_1 \big(\mathcal{H} (\alpha) \big)},
\end{aligned}
\end{flalign}

\noindent where in the latter equality we used the fact that there are $m_i - 1$ choices for the pair $(c_1, c_2)$. 

If $n \ge 2$ and the $(p, d)$ in \eqref{apdalpha} is equal to $(1, 2)$, then $\mu = (\mu_1, \mu_2)$ can be equal to $(m_i, m_j)$ for any $1 \le i \ne j \le n$; in this case, $\overline{\alpha}_1 = \alpha \setminus \{ m_i, m_j \}$. Then, $\mathcal{C}_{2p} (d) = \mathcal{C}_2 (2)$ consists of one element, and so $s = (1, 1)$. Thus, there is one choice for $\kappa = (\kappa_{1, 1}, \kappa_{2, 1}) = (m_1, m_2)$. 

The value of $t$ can either be $0$ or $1$. If it is equal to $0$, then $(c_1, c_2) = (m_i, m_j)$, and if it is equal to $1$, then $(c_1, c_2) = (m_j, m_i)$; in either case, $\mathcal{J} = (1, 2)$ and so $q = 1$. Furthermore, $\alpha_1' = \overline{\alpha} \cup \{ m_i - 1 \} \cup \{ m_j - 1 \}$, meaning that $|\alpha_1'| = |\alpha| - 2$ (since $c_1 + c_2 = m_i + m_j$) and $\ell (\alpha_1') = \ell (\alpha) = n$.

Thus, \eqref{volumeestimate1h1} yields 
\begin{flalign} 
\label{t2}
\begin{aligned}
T_2  & = \displaystyle\frac{1}{4} \displaystyle\frac{1}{\big(|\alpha| + n \big)!} \displaystyle\frac{1}{\nu_1 \big(\mathcal{H} (\alpha) \big)}  \displaystyle\sum_{1 \le i \ne j \le n} \displaystyle\sum_{t = 0}^1 (m_i - 1) (m_j - 1) \big( |\alpha| + n - 2 \big)! \nu_1 \big( \mathcal{H} (\alpha_i') \big) \\
  & = \displaystyle\frac{1}{2}  \displaystyle\frac{1}{\big(|\alpha| + n \big) \big( |\alpha| + n - 1 \big)}  \displaystyle\sum_{1 \le i \ne j \le n}  \displaystyle\frac{(m_i - 1) (m_j - 1)  \nu_1 \big( \mathcal{H} (\alpha_i') \big) }{\nu_1 \big(\mathcal{H} (\alpha) \big)}.
 \end{aligned} 
\end{flalign}

\noindent Now, \eqref{volumeestimate1h1} yields that 
\begin{flalign}
\label{nu1alphanu1alphaestimate2}
 \left| \displaystyle\frac{(m_i - 1) \nu_1 \big( \mathcal{H} (\alpha_i') \big)}{(m_i + 1) \nu_1 \big(\mathcal{H} (\alpha) \big)} - 1 \right| \le \displaystyle\frac{3R}{g}; \qquad \left| \displaystyle\frac{(m_i - 1) (m_j - 1)  \nu_1 \big( \mathcal{H} (\alpha_i') \big) }{(m_1 + 1) (m_2 + 1) \nu_1 \big(\mathcal{H} (\alpha) \big)} - 1 \right| \le \displaystyle\frac{3R}{g}. 
\end{flalign} 

\noindent Inserting the first estimate in \eqref{nu1alphanu1alphaestimate2} into \eqref{t1} then yields  
\begin{flalign}
\label{2t1}
\begin{aligned}
& \left| T_1 - \displaystyle\frac{1}{2} \displaystyle\frac{1}{\big(|\alpha| + n \big) \big( |\alpha| + n - 1 \big)} \displaystyle\sum_{i = 1}^n (m_i - 1) (m_i + 1) \right| \\
& \qquad \qquad \le  \displaystyle\frac{3R}{g} \displaystyle\frac{1}{\big(|\alpha| + n \big) \big( |\alpha| + n - 1 \big)} \displaystyle\sum_{i = 1}^n (m_i - 1) (m_i + 1). 
\end{aligned}
\end{flalign}

\noindent Similarly inserting \eqref{nu1alphanu1alphaestimate2} into \eqref{t2} yields 
\begin{flalign}
\label{2t2}
\begin{aligned}
& \left| T_2 - \displaystyle\frac{1}{2}  \displaystyle\frac{1}{\big(|\alpha| + n \big) \big( |\alpha| + n - 1 \big)}  \displaystyle\sum_{1 \le i \ne j \le n} (m_i + 1) (m_j + 1) \right| \\
& \qquad \qquad \le \displaystyle\frac{3R}{g}  \displaystyle\frac{1}{\big(|\alpha| + n \big) \big( |\alpha| + n - 1 \big)}  \displaystyle\sum_{1 \le i \ne j \le n} (m_i + 1) (m_j + 1).
\end{aligned}
\end{flalign}

\noindent Now the proposition follows from \eqref{2t1}, \eqref{2t2}, and the facts that $T = T_1 + T_2$ and 
\begin{flalign*}
 \Bigg| & \displaystyle\frac{1}{\big(|\alpha| + n \big) \big( |\alpha| + n - 1 \big)}  \bigg( \displaystyle\sum_{i = 1}^n  (m_i - 1) (m_i + 1) + \displaystyle\sum_{1 \le i \ne j \le n} (m_i + 1) (m_j + 1) \bigg) - 1 \Bigg| \\
& \qquad \qquad = \Bigg| \displaystyle\frac{1}{\big(|\alpha| + n \big) \big( |\alpha| + n - 1 \big)} \bigg( \displaystyle\sum_{1 \le i, j \le n} (m_i + 1) (m_j + 1) - 2 \displaystyle\sum_{i = 1}^n (m_i + 1) \bigg) - 1 \Bigg|\\
& \qquad \qquad = \Bigg| \displaystyle\frac{\big( |\alpha| + n \big)^2 - 2 \big( |\alpha| + n \big)}{ \big(|\alpha| + n \big) \big( |\alpha| + n - 1 \big) } -1 \Bigg| =  \displaystyle\frac{1}{ |\alpha| - 1} \le \displaystyle\frac{1}{g}. 
\end{flalign*}
\end{proof}

In \Cref{Termp2} we will establish the following bound on the quantities $\mathcal{A}^{p; d}$. 
\begin{prop} 
	
	\label{xareaestimate} 
	
	There exists a constant $C > 0$ such that the following holds. For any integer $g \ge 3$; partition $\alpha = (m_1, m_2, \ldots , m_n) \in \mathbb{Y}_{2g - 2}$ such that the stratum $\mathcal{H} (\alpha)$ is nonempty and connected; and positive integers $p \ge 2$ and $d \le \max \{ n, 2p \}$, we have that 
	\begin{flalign*}
	\mathcal{A}^{p; d} (\alpha) \le \displaystyle\frac{C^p}{|\alpha|^{2p - 3}}.
	\end{flalign*}
	
\end{prop} 

Assuming \Cref{xareaestimate}, we can now establish \Cref{constantareaasymptotic}.

\begin{proof}[Proof of \Cref{constantareaasymptotic} Assuming \Cref{xareaestimate}]

	We may assume that $g$ is sufficiently large. Thus, throughout this proof, we assume $g > 2^{20} C$, where $C$ denotes the maximum of the two constants $C$ defined in \Cref{ap1estimate} and \Cref{xareaestimate}. 
	
	Then, in  view of \eqref{capd} and \eqref{t}, we have that 
	\begin{flalign*}
	\left| c_{\text{area}} \big(\mathcal{H} (\alpha) \big) - \displaystyle\frac{1}{2} \right| & \le \left| T (\alpha)  - \displaystyle\frac{1}{2} \right| + \displaystyle\sum_{p = 2}^g \displaystyle\sum_{d = 1}^{\min \{ n, 2p \}} \mathcal{A}^{p; d} (\alpha)   \\
	& \le \displaystyle\frac{C}{g} + \displaystyle\sum_{p = 2}^g \displaystyle\sum_{d = 1}^{\min \{ n, 2p \}} \displaystyle\frac{C^p}{|\alpha|^{2p - 3} }  \le \displaystyle\frac{C}{g} + \displaystyle\frac{2}{g} \displaystyle\sum_{p = 2}^g \displaystyle\frac{p C^p}{g^{p - 2} } \le \displaystyle\frac{2C^3}{g},
	\end{flalign*}
	
	\noindent where we have used \Cref{ap1estimate} and \Cref{xareaestimate} to deduce the second estimate, and the fact that $|\alpha| \ge g$ to deduce the third estimate. This implies the theorem. 
\end{proof}

\subsection{Proof of Proposition \ref{xareaestimate}} 

\label{Termp2} 

In this section we establish \Cref{xareaestimate}. Throughout this section, we abbreviate $\mathcal{A} = \mathcal{A}^{p; d} = \mathcal{A}^{(p; d)} (\alpha)$. 

We begin with the following lemma, which provides a preliminary estimate on $\mathcal{A}$. To state this bound, we recall some notation from the proof of \Cref{cpm1m2estimate}. Specifically, for any nonnegative composition $D = (D_1, D_2, \ldots , D_p) \in \mathcal{G}_{n - d} (p)$, let $\mathfrak{N} (\overline{\alpha}; D) \subseteq \mathfrak{N}_p (\overline{\alpha})$ denote the family of set partitions $(\overline{\alpha}_1, \overline{\alpha}_2, \ldots , \overline{\alpha}_p)$ of $\overline{\alpha}$ such that $\ell (\overline{\alpha}_i) = D_i$ for each $i \in [1, p]$. Furthermore, let $R > 2$ denote the (sufficiently large) constant from \eqref{volumeestimate}. 

\begin{lem}
	
\label{1apdestimate} 

Adopting the notation of \Cref{xareaestimate}, we have that 
\begin{flalign}
\label{apdestimate1} 
\begin{aligned} 
\mathcal{A} & \le  p R^{p + 2} \displaystyle\sum_{\mu \in \mathfrak{M}_d (\alpha)}  \displaystyle\sum_{D \in \mathcal{G}_{n - d} (p)} \displaystyle\sum_{(\overline{\alpha}_i ) \in \mathfrak{N} (\overline{\alpha}; D) }  \displaystyle\sum_{B \in \mathcal{G}_{|\mu| - 2p} (p)}  \displaystyle\frac{ \prod_{i = 1}^p \big( |\overline{\alpha}_i| + B_i + D_i + 3 \big)!}{\big(|\alpha| + n \big)!} \displaystyle\prod_{i = 1}^d (\mu_i + 1),
\end{aligned}
\end{flalign}
	
\noindent where, in the sum on the right side of \eqref{apdestimate1}, we denoted the nonnegative compositions $B = (B_1, B_2, \ldots , B_p) \in \mathcal{G}_{|\mu| - 2p} (p)$ and $D = (D_1, D_2, \ldots , D_p) \in \mathcal{G}_{n - d} (p)$. 

\end{lem} 

\begin{proof}

First observe, by \eqref{volumeestimate}, that
\begin{flalign*}
\displaystyle\frac{1}{\nu_1 \big( \mathcal{H} (\alpha) \big)} \displaystyle\prod_{i = 1}^p \nu_1 \big( \mathcal{H} (\alpha_i') \big)  \displaystyle\prod_{\substack{1 \le i \le p \\ 2i \in \mathcal{J}}} c_{2i - 1} c_{2i} \displaystyle\prod_{\substack{1 \le i \le p \\ 2i \notin \mathcal{J}}} (c_{2i - 1} + c_{2i} - 1)  \le R^{p + 1} \displaystyle\prod_{i = 1}^d (\mu_i + 1), 
\end{flalign*}

\noindent which upon insertion into \eqref{apdalpha} yields that 
\begin{flalign*}
\mathcal{A} & \le  \displaystyle\frac{R^{p + 1}}{\big(|\alpha| + n \big)!} \displaystyle\sum_{\mu \in \mathfrak{M}_d (\alpha)} \displaystyle\sum_{(\overline{\alpha}_i ) \in \mathfrak{N}_p (\overline{\alpha}) }  \displaystyle\sum_{s \in \mathcal{C}_{2p} (d)} \displaystyle\sum_{\kappa^{(i)} \in \mathcal{C}_{\mu_i} (s_i)} \displaystyle\sum_{t = 0}^{2p - 1} q \displaystyle\prod_{i = 1}^p \big( |\alpha_i'| + \ell (\alpha_i') \big)! \displaystyle\prod_{i = 1}^d (\mu_i + 1) \\
& \le  \displaystyle\frac{p R^{p + 2}}{\big(|\alpha| + n \big)!} \displaystyle\sum_{\mu \in \mathfrak{M}_d (\alpha)} \displaystyle\sum_{(\overline{\alpha}_i ) \in \mathfrak{N}_p (\overline{\alpha}) }  \displaystyle\sum_{s \in \mathcal{C}_{2p} (d)} \displaystyle\sum_{\kappa^{(i)} \in \mathcal{C}_{\mu_i} (s_i)} \displaystyle\max_{t \in [0, 2p - 1]}   \displaystyle\prod_{i = 1}^p \big( |\alpha_i'| + \ell (\alpha_i') \big)! \displaystyle\prod_{i = 1}^d (\mu_i + 1),
\end{flalign*}

\noindent where in the latter estimate we used the fact that $q \le 2p$. 

Next, since $\mathfrak{N}_p (\overline{\alpha}) = \bigcup_{D \in \mathcal{G}_{n - d} (p)} \mathfrak{N} (\overline{\alpha}; D)$, it follows that 
\begin{flalign*}
\mathcal{A}  \le  \displaystyle\frac{p R^{p + 2}}{\big(|\alpha| + n \big)!} \displaystyle\sum_{\mu \in \mathfrak{M}_d (\alpha)} \displaystyle\sum_{D \in \mathcal{G}_{n - d} (p)} \displaystyle\sum_{(\overline{\alpha}_i ) \in \mathfrak{N} (\overline{\alpha}; D) }  \displaystyle\sum_{s \in \mathcal{C}_{2p} (d)} \displaystyle\sum_{\kappa^{(i)} \in \mathcal{C}_{\mu_i} (s_i)} & \displaystyle\max_{t \in [0, 2p - 1]}   \displaystyle\prod_{i = 1}^p \big( |\alpha_i'| + \ell (\alpha_i') \big)! \\
& \qquad \quad \times \displaystyle\prod_{i = 1}^d (\mu_i + 1).
\end{flalign*}	

\noindent Thus, denoting $D = (D_1, D_2, \ldots , D_p)$ and using the facts that $|\alpha_i'| = |\overline{\alpha}_i| + c_{2i - 1} + c_{2i} - 2$ and $\ell (\alpha_i') \le D_i + 2$ for $(\overline{\alpha}_i) \in \mathfrak{N} (\overline{\alpha}, D)$ (since $\ell (\overline{\alpha}_i) = D_i$), we deduce that 
\begin{flalign*}
\mathcal{A} & \le  p R^{p + 2} \displaystyle\sum_{\mu \in \mathfrak{M}_d (\alpha)} \displaystyle\prod_{i = 1}^d (\mu_i + 1) \\
& \qquad \qquad \times \displaystyle\sum_{D \in \mathcal{G}_{n - d} (p)} \displaystyle\sum_{(\overline{\alpha}_i ) \in \mathfrak{N} (\overline{\alpha}; D) }  \displaystyle\sum_{s \in \mathcal{C}_{2p} (d)} \displaystyle\sum_{\kappa^{(i)} \in \mathcal{C}_{\mu_i} (s_i)} \displaystyle\max_{t \in [0, 2p - 1]}  \displaystyle\frac{ \prod_{i = 1}^p \big( |\overline{\alpha}_i| + D_i + c_{2i - 1} + c_{2i} \big)!}{\big(|\alpha| + n \big)!} \\
& =  p R^{p + 2} \displaystyle\sum_{\mu \in \mathfrak{M}_d (\alpha)} \displaystyle\prod_{i = 1}^d (\mu_i + 1) \\
& \qquad \qquad \times \displaystyle\sum_{D \in \mathcal{G}_{n - d} (p)} \displaystyle\sum_{(\overline{\alpha}_i ) \in \mathfrak{N} (\overline{\alpha}; D) }  \displaystyle\sum_{s \in \mathcal{C}_{2p} (d)} \displaystyle\sum_{\kappa^{(i)} \in \mathcal{C}_{\mu_i} (s_i)} \displaystyle\max_{t \in [0, 2p - 1]}  \displaystyle\frac{ \prod_{i = 1}^p \big( |\overline{\alpha}_i| + B_i + D_i + 2 \big)!}{\big(|\alpha| + n \big)!}, 
\end{flalign*}

\noindent where in the last equality we denoted $B_i = c_{2i - 1} + c_{2i} - 2 \ge 0$. 

Now, define the nonnegative composition $B = (B_1, B_2, \ldots , B_p) \in \mathcal{G}_{|\mu| - 2p} (p)$. Observe that we can upper bound the sum over all sequences of compositions $s = (s_1, s_2, \ldots , s_d) \in \mathcal{C}_{2p} (d)$ and  $\kappa^{(i)} \in \mathcal{C}_{\mu_i} (s_i)$ above by instead summing over all nonnegative compositions $B \in \mathcal{G}_{|\mu| - 2p} (p)$ and then summing over all $2p$-tuples of positive integers $c = (c_1, c_2, \ldots , c_{2p})$ such that $c_{2i - 1} + c_{2i} = B_i + 2$ for each $i \in [1, p]$. This yields 
\begin{flalign*}
\mathcal{A} & \le  p R^{p + 2} \displaystyle\sum_{\mu \in \mathfrak{M}_d (\alpha)} \displaystyle\prod_{i = 1}^d (\mu_i + 1) \\
& \qquad \qquad \times \displaystyle\sum_{D \in \mathcal{G}_{n - d} (p)} \displaystyle\sum_{(\overline{\alpha}_i ) \in \mathfrak{N} (\overline{\alpha}; D) }  \displaystyle\sum_{B \in \mathcal{G}_{|\mu| - 2p} (p)} \displaystyle\sum_{c_{2i - 1} + c_{2i} = B_i + 2}  \displaystyle\frac{ \prod_{i = 1}^p \big( |\overline{\alpha}_i| + B_i + D_i + 2 \big)!}{\big(|\alpha| + n \big)!}.
\end{flalign*}

Next observe that, for any fixed nonnegative composition $B \in \mathcal{G}_{|\mu| - 2p} (p)$, there are at most $B_i + 1$ choices for the pair $(c_{2i - 1}, c_{2i})$. Therefore, for any fixed $B \in \mathcal{G}_{|\mu| - 2p} (p)$, there are at most $\prod_{i = 1}^p (B_i + 1)$ choices for $c$. This implies that  
\begin{flalign*}
\mathcal{A} & \le  p R^{p + 2} \displaystyle\sum_{\mu \in \mathfrak{M}_d (\alpha)} \displaystyle\prod_{i = 1}^d (\mu_i + 1) \\
& \qquad  \times \displaystyle\sum_{D \in \mathcal{G}_{n - d} (p)} \displaystyle\sum_{(\overline{\alpha}_i ) \in \mathfrak{N} (\overline{\alpha}; D) }  \displaystyle\sum_{B \in \mathcal{G}_{|\mu| - 2p} (p)}  \displaystyle\frac{ \prod_{i = 1}^p \big( |\overline{\alpha}_i| + B_i + D_i + 2 \big)! }{\big(|\alpha| + n \big)!} \displaystyle\prod_{i = 1}^p (B_i + 1),
\end{flalign*}

\noindent from which we deduce the lemma in view of the fact that $\big( |\overline{\alpha}_i| + B_i + D_i + 2 \big)! (B_i + 1) \le \big( |\overline{\alpha}_i| + B_i + D_i + 3 \big)!$ for each $i \in [1, p]$. 
\end{proof}

Now we can establish \Cref{xareaestimate}, whose proof is now similar to the end of that of \Cref{cpm1m2estimate}.  

\begin{proof}[Proof of \Cref{xareaestimate}] 

We first bound the maximum of the product on the right side of \eqref{apdestimate1} over all $(\overline{\alpha}_i ) \in \mathfrak{N} (\overline{\alpha}; D)$. To that end, we apply \Cref{aibiproduct} with the $A_i$ there equal to the $|\overline{\alpha}_i|$ here; the $B_i$ there equal to the $B_i + D_i$ here; the $C_i$ there equal to the $D_i$ here; and the $d$ there equal to $3$ here. Since 
\begin{flalign*}
\displaystyle\sum_{i = 1}^p |\overline{\alpha}_i| = |\alpha| - |\mu|; \qquad \displaystyle\sum_{i = 1}^p B_i = |\mu| - 2p; \qquad \displaystyle\sum_{i = 1}^p D_i = n - d,
\end{flalign*}

\noindent the $T$ from that lemma is equal to $|\alpha| - |\mu|$; the $U$ from that lemma is equal to $|\mu| + n - d - 2p$; and the $V$ from that lemma is equal to $n - d$. Thus, we deduce from \eqref{aibidtuvproduct} that 
\begin{flalign*}
\displaystyle\prod_{i = 1}^p \big( |\overline{\alpha}_i| + B_i + D_i + 3 \big)! \le \displaystyle\frac{\big( |\alpha| + n - d - 2p + 3 \big)!}{\big( 2n + |\mu| - 2d - 2p + 3 \big)! } \displaystyle\prod_{i = 1}^p (B_i + 2 D_i + 3)!,
\end{flalign*}

\noindent which upon insertion into \eqref{apdestimate1} yields 
\begin{flalign*}
\mathcal{A} & \le p R^{p + 2}  \displaystyle\sum_{\mu \in \mathfrak{M}_d (\alpha)} \displaystyle\prod_{i = 1}^d (\mu_i + 1) \\
& \qquad \times \displaystyle\sum_{B \in \mathcal{G}_{|\nu| - 2p} (p)}   \displaystyle\sum_{D \in \mathcal{G}_{n - d} (p)} \displaystyle\sum_{(\overline{\alpha}_i) \in \mathfrak{N} (\overline{\alpha}; D)} \displaystyle\frac{ \big( |\alpha| + n - d - 2p + 3 \big)!}{\big( 2n + |\mu| - 2d - 2p + 3 \big)! \big(|\alpha| + n \big)!}  \displaystyle\prod_{i = 1}^p (B_i + 2 D_i + 3)! \\
& = p R^{p + 2}  \displaystyle\sum_{\mu \in \mathfrak{M}_d (\alpha)} \displaystyle\sum_{B \in \mathcal{G}_{|\nu| - 2p} (p)}   \displaystyle\sum_{D \in \mathcal{G}_{n - d} (p)}  \displaystyle\frac{ \big( |\alpha| + n - d - 2p + 3 \big)!}{\big( 2n + |\mu| - 2d - 2p + 3 \big)! \big(|\alpha| + n \big)!} \binom{n - d}{D_1, D_2, \ldots , D_p} \\
& \qquad \qquad \qquad \qquad \qquad \qquad \qquad \qquad  \times \displaystyle\prod_{i = 1}^p (B_i + 2 D_i + 3)!  \displaystyle\prod_{i = 1}^d (\mu_i + 1) , 
\end{flalign*}

\noindent where in the equality we used the fact that $\big| \mathfrak{N} (\overline{\alpha}; D) \big| = \binom{n - d}{D_1, D_2, \ldots , D_p}$, which holds due to the first identity in \eqref{aipksize}. It follows that 
\begin{flalign*}
\mathcal{A} & \le p R^{p + 2}  \displaystyle\sum_{\mu \in \mathfrak{M}_d (\alpha)} \displaystyle\sum_{B \in \mathcal{G}_{|\nu| - 2p} (p)}   \displaystyle\sum_{D \in \mathcal{G}_{n - d} (p)}  \displaystyle\frac{ \big( |\alpha| + n - d - 2p + 3 \big)! (n - d)!}{\big( 2n + |\mu| - 2d - 2p + 3 \big)! \big(|\alpha| + n \big)!} \\
& \qquad \qquad \qquad \qquad \qquad \qquad \qquad \qquad  \times \displaystyle\prod_{i = 1}^p \binom{B_i + 2 D_i + 3}{D_i}  (B_i + D_i + 3)!  \displaystyle\prod_{i = 1}^d (\mu_i + 1). 
\end{flalign*}

\noindent Since 
\begin{flalign*}
\binom{B_i + 2 D_i + 3}{D_i} = \displaystyle\frac{(B_i + 2 D_i + 3) (B_i + 2 D_i + 2) (B_i + 2 D_i + 1)}{(B_i + D_i + 3) (B_i + D_i + 2) (B_i + D_i + 1)} \binom{B_i + 2 D_i}{D_i} \le 8 \binom{B_i + 2 D_i}{D_i},
\end{flalign*} 

\noindent we obtain that 
\begin{flalign}
\label{apdestimate2}
\begin{aligned} 
\mathcal{A} & \le p (8 R)^{p + 2}  \displaystyle\sum_{\mu \in \mathfrak{M}_d (\alpha)} \displaystyle\sum_{B \in \mathcal{G}_{|\nu| - 2p} (p)}   \displaystyle\sum_{D \in \mathcal{G}_{n - d} (p)}  \displaystyle\frac{ \big( |\alpha| + n - d - 2p + 3 \big)! (n - d)!}{\big( 2n + |\mu| - 2d - 2p + 3 \big)! \big(|\alpha| + n \big)!} \\
& \qquad \qquad \qquad \qquad \qquad \qquad \qquad \qquad  \times \displaystyle\prod_{i = 1}^p \binom{B_i + 2 D_i}{D_i}  (B_i + D_i + 3)!  \displaystyle\prod_{i = 1}^d (\mu_i + 1). 
\end{aligned}
\end{flalign} 

Next, let us apply \Cref{sumaijaiestimate} with the $r$ there equal to $2$ here; the $A_i$ there equal to the $B_i + 2 D_i$ here; the $A_{i, 1}$ there equal to the $D_i$ here; and the $A_{i, 2}$ there equal to the $B_i + D_i$ here. Since $\sum_{i = 1}^p D_i = n - d$ and $\sum_{i = 1}^p B_i = |\mu| - 2p$, it follows that
\begin{flalign*}
\displaystyle\prod_{i = 1}^p \binom{B_i + 2 D_i}{D_i} = \binom{2n + |\mu| - 2d - 2p}{n - d},
\end{flalign*} 

\noindent which upon insertion into \eqref{apdestimate2} yields 
\begin{flalign}
\label{apdestimate3} 
\begin{aligned} 
\mathcal{A} & \le p (8 R)^{p + 2}  \displaystyle\sum_{\mu \in \mathfrak{M}_d (\alpha)}  \displaystyle\frac{ \big( |\alpha| + n - d - 2p + 3 \big)! \big( 2n + |\mu| - 2d - 2p \big)!}{\big( 2n + |\mu| - 2d - 2p + 3 \big)! \big( n + |\mu| - d - 2p \big)! \big(|\alpha| + n \big)!} \displaystyle\prod_{i = 1}^d (\mu_i + 1) \\
& \qquad \qquad \qquad  \times \displaystyle\sum_{B \in \mathcal{G}_{|\mu| - 2p} (p)}   \displaystyle\sum_{D \in \mathcal{G}_{n - d} (p)} \displaystyle\prod_{i = 1}^p   (B_i + D_i + 3)!. 
\end{aligned}
\end{flalign} 

\noindent In order to bound the last two sums on the right side of \eqref{apdestimate3}, we apply \Cref{2sumcompositions} with the $r$ there equal to $p$ here; the $A_i$ there equal to the $B_i$ here; the $B_i$ there equal to the $D_i$ here; the $U$ there equal to $n - d$ here; and the $V$ there equal to $|\mu| - 2p$ here, we obtain that 
\begin{flalign*}
\mathcal{A} & \le p (2^{14} R)^{p + 2}  \displaystyle\sum_{\mu \in \mathfrak{M}_d (\alpha)} \displaystyle\frac{\big( |\alpha| + n - d - 2p + 3 \big)! }{\big(|\alpha| + n \big)!}  \displaystyle\frac{ \big( n + |\mu| - d - 2p + 3 \big)!}{\big( n + |\mu| - d - 2p \big)! } \displaystyle\frac{ \big( 2n + |\mu| - 2d - 2p \big)!}{\big( 2n + |\mu| - 2d - 2p + 3 \big)! } \\
& \qquad \qquad \qquad \quad \times \displaystyle\prod_{i = 1}^d (\mu_i + 1) \\
& \le p (2^{14} R)^{p + 2}  \displaystyle\frac{\big( |\alpha| + n - d - 2p + 3 \big)! }{\big(|\alpha| + n \big)!} \displaystyle\sum_{\mu \in \mathfrak{M}_d (\alpha)}  \displaystyle\prod_{i = 1}^d (\mu_i + 1).
\end{flalign*} 

\noindent Using the fact that $\sum_{\mu \in \mathfrak{M}_d (\alpha)} \prod_{i = 1}^d (\mu_i + 1) \le \prod_{i = 1}^d \sum_{i = 1}^n (m_i + 1) = \big( |\alpha| + n \big)^d$, it follows that 
\begin{flalign}
\label{apdestimate4} 
\begin{aligned} 
\mathcal{A} & \le p (2^{14} R)^{p + 2}  \displaystyle\frac{\big( |\alpha| + n - d - 2p + 3 \big)! \big( |\alpha| + n \big)^d }{\big(|\alpha| + n \big)!}. 
\end{aligned}
\end{flalign} 

\noindent Now the proposition follows from \eqref{apdestimate4}, the fact that $d \le 2p$, and the facts that 
\begin{flalign*}
& \big( |\alpha| + n - d - 2p + 3 \big)! \le 4 e^{d + 2p - |\alpha| - n - 3} \big( |\alpha| + n \big)^{|\alpha| + n - d - 2p + 7 / 2}; \\
& \big( |\alpha| + n \big)! \ge 2 e^{- |\alpha| - n} \big( |\alpha| + n \big)^{|\alpha| + n + 1 / 2},
\end{flalign*}

\noindent which follow from the second estimate in \eqref{2ll}. 
\end{proof}


\begin{thebibliography}{}
	
\bibitem{LGAVSD} \label{LGAVSD} A. Aggarwal, Large Genus Asymptotics for Volumes of Strata of Abelian Differentials, With an Appendix by A. Zorich, preprint, https://arxiv.org/abs/1804.05431. 

\bibitem{SSRCMS} \label{SSRCMS} D. Chen, Square-Tiled Surfaces and Rigid Curves on Moduli Spaces, \emph{Adv. Math.} \textbf{228}, 1135--1162, 2011. 

\bibitem{QLGL} \label{QLGL} D. Chen, M. M\"{o}ller, and D. Zagier, Quasimodularity and Large Genus Limits of Siegel--Veech Constants, \textit{J. Amer. Math. Soc.} \textbf{31}, 1059--1163, 2018.  
	
\bibitem{VSCC} \label{VSCC} D. Chen, M. M\"{o}ller, A. Sauvaget, and D. Zagier, Masur--Veech Volumes and Intersection Theory on Moduli Spaces of Abelian Differentials, preprint, https://arxiv.org/abs/1901.01785. 
	
\bibitem{ERG} \label{ERG} A. Eskin, M. Kontsevich, and A. Zorich, Sum of Lyapunov Exponents of the Hodge Bundle With Respect to the Teichm\"{u}ller Geodesic Flow, \textit{Publ. Math. IHES} \textbf{120}, 207--333, 2014. 
	

\bibitem{AFS} \label{AFS} A. Eskin and H. Masur, Asymptotic Formulas on Flat Surfaces, \emph{Ergod. Th. Dynam. Sys.} \textbf{21}, 443--478, 2001.

\bibitem{PBC} \label{PBC} A. Eskin, H. Masur, and A. Zorich, Moduli Spaces of Abelian Differentials: The Principal Boundary, Counting Problems, and the Siegel--Veech Constants, \textit{Publ. Math. IHES} \textbf{97}, 61--179, 2003. 


\bibitem{ANBCTV} \label{ANBCTV} A. Eskin and A. Okounkov, Asymptotics of Numbers of Branched Coverings of a Torus and Volumes of Moduli Spaces of Holomorphic Differentials, \textit{Invent. Math.} \textbf{145}, 59--103, 2001. 

\bibitem{TCBC} \label{TCBC} A. Eskin, A. Okounkov, and R. Pandharipande, The Theta Characteristic of a Branched Covering, \emph{Adv. Math.} \textbf{217}, 873--888, 2008.

\bibitem{VSDCLG} \label{VSDCLG} A. Eskin and A. Zorich, Volumes of Strata of Abelian Differentials and Siegel--Veech Constants in Large Genera, \textit{Arnold Math. J.} \textbf{1}, 481--488, 2015.

\bibitem{TADIETFSB} \label{TADIETFSB} G. Forni and C. Matheus, Introduction to Teichm\"{u}ller Theory and its Applications to Dynamics of Interval Exchange Transformations, Flows on Surfaces and Billiards, \emph{J. Mod. Dyn.} \textbf{8}, 271--436, 2014. 


\bibitem{CCMS} \label{CCMS} M. Kontsevich and A. Zorich, Connected Components of the Moduli Spaces of Abelian Differentials With Prescribed Singularities, \textit{Invent. Math.} \textbf{153}, 631--678, 2003. 


\bibitem{ETMF} \label{ETMF} H. Masur, Interval Exchange Transformations and Measured Foliations, \textit{Ann. Math.} \textbf{115}, 169--200, 1982. 	

\bibitem{TSGTMS} \label{TSGTMS} H. Masur, K. Rafi, and A. Randecker, The Shape of a Generic Translation Surface, preprint, https://arxiv.org/abs/1809.10769. 

\bibitem{GVRHSLG} \label{GVRHSLG} M. Mirzakhani, Growth of Weil--Petersson Volumes and Random Hyperbolic Surfaces of Large Genus, \emph{J. Differential Geom.} \textbf{94}, 267--300, 2013. 

\bibitem{LGAIMSC} \label{LGAIMSC} M. Mirzakhani and P. Zograf, Towards Large Genus Asymptotics of Intersection Numbers on Moduli Spaces of Curves, \emph{Geom. Funct. Anal.} \textbf{25}, 1258--1289, 2015. 


\bibitem{VSI} \label{VSI}	A. Sauvaget, Volumes and Siegel--Veech Constants of $\mathcal{H} (2g - 2)$ and Hodge Integrals, \emph{Geom. Funct. Anal.} \textbf{28}, 1756--1779, 2018. 

\bibitem{CCSD} \label{CCSD} A. Sauvaget, Cohomology Classes of Strata of Differentials, preprint, https://arxiv.org/abs/1701.07867. 

\bibitem{MTSIEM} \label{MTSIEM} W. A. Veech, Gauss Measures for Transformations on the Space of Interval Exchange Maps, \textit{Ann. Math.} \textbf{115}, 201--242, 1982. 

\bibitem{M} \label{M} W. A. Veech, Siegel Measures, \emph{Ann. Math.} \textbf{148}, 895--944, 1998. 

\bibitem{PGTS} \label{PGTS}	Ya. Vorobets, \textit{Periodic Geodesics of Translation Surfaces}, (2003), in S. Kolyada, Yu. I. Manin and T. Ward (eds.) Algebraic and Topological Dynamics, Contemporary Math., vol. 385, pp. 205--258, Amer. Math. Soc., Providence, 2005.

\bibitem{TSOC} \label{TSOC} A. Wright, Translation Surfaces and Their Orbit Closures: An Introduction for a Broad Audience, \emph{EMS Surv. Math. Sci.} \textbf{2}, 63--108, 2015. 

\bibitem{LGAV} \label{LGAV} P. Zograf, On the Large Genus Asymptotics of Weil-Petersson Volumes, preprint, https://arxiv.org/abs/0812.0544. 


\bibitem{FS} \label{FS} A. Zorich, Flat Surfaces, In: \emph{Frontiers in Number Theory, Physics, and Geometry} (P. E. Cartier, B. Julia, P. Moussa, and P. Vanhove ed.), Springer, Berlin, 437--583, 2006.

\bibitem{SVMS} \label{SVMS} A. Zorich, Square Tiled Surfaces and Teichm\"{u}ller Volumes of the Moduli Spaces of Abelian Differentials, In: \emph{Rigidity in Dynamics and Geometry} (M. Burger and A. Iozzi), Springer, Berlin, 459--471, 2002. 



\end{thebibliography}
\end{document}